\documentclass[a4paper,11pt]{amsart}
\usepackage{amsmath}
\usepackage{url,graphicx}
\usepackage{amssymb, color, pstricks-add, manfnt}
\usepackage{amsmath, amsthm,  amsfonts}
\usepackage{mathrsfs}
\usepackage{cite}
\usepackage{enumerate}
\usepackage[none]{hyphenat}
\usepackage{multirow}
\usepackage{amsfonts}

\theoremstyle{plain}
\newtheorem{Theorem}{Theorem}[section]
\newtheorem{Lemma}{Lemma}[section]
\newtheorem{Proposition}{Proposition}[section]

\newtheorem{Corollary}{Corollary}[section]
\newtheorem{Definition}{Definition}[section]

\newtheorem{TheoremA} {Theorem}

\theoremstyle{remark}
\newtheorem{remark}{Remark}

\numberwithin{equation}{section}
\numberwithin{figure}{section}
\numberwithin{remark}{section}
\usepackage{autobreak}%让长公式自动换行
\allowdisplaybreaks%跨页自动断页

\usepackage{hyperref}
\usepackage{cite}

\begin{document}
\begin{sloppypar}

	\title{Regularity of solutions to degenerate normalized $p$-Laplacian equation with general variable exponents}
	
	\author{Jiangwen Wang}
	\address{School of Mathematics and Shing-Tung Yau Center of Southeast University, Southeast University, Nanjing 211189, P.R. China}
	\email{jiangwen\_wang@seu.edu.cn} 	

    \author{Yunwen Yin}
    \address{School of Mathematics, Southeast University, Nanjing 211189, P.R. China}
    \email{yunwenyin@seu.edu.cn}

	\author{Feida Jiang$^*$}
	\address{School of Mathematics and Shing-Tung Yau Center of Southeast University, Southeast University, Nanjing 211189, P.R. China}
	\email{jiangfeida@seu.edu.cn}	
	
	\subjclass[2010]{35A01; 35D40; 35J70; 35J92}

	%\communicated{}
	\date{\today}
	\thanks{*corresponding author}
	
		\keywords{normalized $p$-Laplacian operator, degenerate elliptic equation, viscosity solutions, free transmission}

	\begin{abstract}
	In this paper, we consider degenerate quasilinear elliptic models of normalized $ p$-Laplacian type. We establish local $C^{1,\alpha'}$ regularity of viscosity solutions by making use of the compactness argument, scaling techniques and the localized oscillating method. In addition, we also obtain almost optimal pointwise $C^{1,\tau} $ regularity for degenerate free transmission problem related to normalized $ p$-Laplacian. Our argument is based on a new improved oscillation-type estimate combined with a localized analysis.
	\end{abstract}

	\maketitle		
%\tableofcontents

\section{Introduction}
In this paper, we are concerned with local regularity properties for solutions of the degenerate normalized $ p $-Laplacian equation with general variable exponents
\begin{align}\label{intro:eq1}
-\bigg\{|Du|^{\alpha(x,u)}+a(x)|Du|^{\beta(x,u)}\bigg\}\Delta_{p}^{\rm{N}} u =f(x) \ \text{in} \ B_{1}\subset \mathbb{R}^{n},
\end{align}
where $ n\geq 2$, $ 0 \leq a(x)\in C(B_{1}) $, $ 0 <a_{1} \leq \alpha(x,u) \leq \beta(x,u)\leq a_{2} <\infty $, $ 1 <p<\infty $ and $ \Delta_{p}^{\rm{N}} $ denotes the normalized $ p$-Laplacian operator given by
\begin{equation}\label{introd:eq1}
\Delta_{p}^{\rm{N}} u:= \Delta u +(p-2)\bigg \langle D^{2}u\frac{Du}{|Du|},\frac{Du}{|Du|} \bigg \rangle .
 \end{equation}

The regularity for solutions to the equation \eqref{intro:eq1} has attracted much attention along the last years due to intrinsic with Tug-of-War games, as well as related topics like double-phase models in divergence form. We shall first list some of the regularity results for equation \eqref{intro:eq1} in different situations of $\alpha(x,u)$ and $a(x)$ as follows.

When $ \alpha(x,u):=p-2, a(x)=0 $, equation \eqref{intro:eq1} corresponds to the classical $ p$-Laplacian equation
$$ -\Delta_{p}u:=-\text{div}(|Du|^{p-2}Du)=f(x) \ \ \text{in} \ B_{1},  $$
whose regularity of gradient was studied in \cite{KU77, W94}.

When $ \alpha(x,u):=0, a(x)=0 $, equation \eqref{intro:eq1} corresponds to the normalized $ p$-Laplacian equation $ -\Delta_{p}^{\rm{N}} u=f $. This type of equation is closely related to the stochastic tug-of-war games \cite{PS08,MP10,BG15} and the image processing problems \cite{Does}. Interior H\"{o}lder gradient regularity for viscosity solutions to this equation was analyzed by Attouchi et al in \cite{AME17} by using an improvement-of-flatness approach. Moreover, in the case when $ p>2 $, under a weaker norm of $ f $, the authors also got the same regularity result in \cite{AME17} employing divergence form theory and nonlinear potential theory. Later, Banerjee and Munive \cite{BM20} proved gradient continuity estimates for viscosity solutions of $ -\Delta_{p}^{\rm{N}} u=f $ in terms of the scaling critical $ L(n,1) $-norm of $ f $(a Lorentz space), which improved the regularity result in \cite{AME17}.

When $\alpha(x,u):=\gamma>-1 ,a(x)=0 $, equation \eqref{intro:eq1} is related to the degenerate or singular normalized $ p$-Laplacian equation
\begin{equation}\label{intro:eq2}
   -|Du|^{\gamma}\Delta_{p}^{\rm{N}}u:=-|Du|^{\gamma}\bigg(\Delta u +(p-2)\bigg \langle D^{2}u\frac{Du}{|Du|},\frac{Du}{|Du|} \bigg \rangle \bigg)=f \ \text{in} \ B_{1} .
\end{equation}
In the restricted case $ p\geq 2 $ and $ -1<\gamma \leq 0 $, Birindelli and Demengel \cite{IF10} showed the local H\"{o}lder regularity of the gradient for solutions to \eqref{intro:eq2} by using  approximations and a fixed point argument. Later, the result was extended to the full range $ \gamma >-1 $ and $ p>1 $ in \cite{AE18}. For a fully nonlinear operator instead of $ \Delta_{p}^{\rm{N}}  $ in \eqref{intro:eq1}, the first H\"{o}lder gradient regularity for a viscosity solution $ u $ of
\begin{equation*}\label{intro:eq3}
 |Du|^{\gamma}F(D^{2}u) = f(x)
\end{equation*}
was established in \cite{LS13} when $\gamma >0 $ and $ F $ is a uniformly elliptic operator. Since then, there are many papers concerning fully nonlinear elliptic equations with generalized degeneracy or singularity, we refer to \cite{De21, YRZ21, SSKS221, AEGE21} for results on interior regularity, \cite{DV21, DRRV23, 2DV21} for free boundary problems, and \cite{BDR23} for singularly perturbed problem. It should be noted that $ \Delta_{p}^{\rm{N}} u $ is, in general, discontinuity at the set $ \{Du=0\}$, and therefore the results of \cite{De21, YRZ21, SSKS221, AEGE21} cannot be applied directly to the operator $ \Delta_{p}^{\rm{N}}  $.

Motivated by the results in \cite{AE18, LS13, De21, YRZ21, SSKS221, AEGE21}, we naturally consider the degenerate normalized $ p$-Laplacian equation with general variable exponents of the form \eqref{intro:eq1}. To the best of our knowledge, gradient H\"{o}lder regularity of viscosity solutions for \eqref{intro:eq1} is unknown in the current literature. By making use of the compactness argument, scaling techniques and the localized oscillating method, we show that the viscosity solutions for \eqref{intro:eq1} are locally of class $ C^{1,\alpha'}(B_{1}) $ in this paper.

Before stating our main result, we make some basic assumptions. We first assume that
\begin{equation}\label{pre:eq1}
1<p<\infty .
\end{equation}
Concerning the nonhomogeneous degenerate terms in (\ref{intro:eq1}), we shall require that the general variable exponents $ \alpha(x,u), \beta(x,u)$ fulfill
\begin{equation}\label{pre:eq2}
0 <a_{1} \leq \alpha(x,u) \leq \beta(x,u)\leq a_{2} <\infty
\end{equation}
for some positive constants $a_1$ and $a_2$, and the modulating coefficient $ a(x) $ satisfies
\begin{equation}\label{pre:eq3}
0 \leq a(x)\in C(B_{1})
\end{equation}
and the source term $ f $ satisfies
\begin{equation}\label{pre:eq4}
f \in C(B_{1}) \cap L^{\infty}(B_{1}).
\end{equation}

Now we state our first main result.

\begin{Theorem}(Local $ C^{1, \alpha'}$ regularity)
\label{thm1}
Under the assumptions \eqref{pre:eq1}--\eqref{pre:eq4}, suppose that $ u $ is a viscosity solution to \eqref{intro:eq1}, then there exists $ \alpha' =\alpha'(p,n,a_{1},a_{2})\in \big(0,\frac{1}{1+a_{2}}\big] $ such that for any ball $ B \subset\subset B_{1} $, we have $ u \in C^{1,\alpha'}(B) $ with the estimate
\begin{equation}
\label{eq4}
||u||_{C^{1,\alpha'}(B)} \leq C\bigg( ||u||_{L^{\infty}(B_{1})}+\max \{||f||_{L^{\infty}(B_{1})}, ||f||_{L^{\infty}(B_{1})}^{\frac{1}{1+a_{1}}}, ||f||_{L^{\infty}(B_{1})}^{\frac{1}{1+a_{2}}} \} \bigg),
\end{equation}
where $ C $ is a positive constant depending only on $ n, p, a_{1}$, $a_{2} $ and $  \text{dist}(B, \partial B_{1})^{-1-\alpha'} $.
\end{Theorem}

We remark that in Theorem \ref{thm1} we only require that $\alpha(x,u)$ and $\beta(x,u)$ have positive lower and upper bounds. Due to the generality of degeneracy term $ \alpha(x,u) $ and $ \beta(x,u) $, our main result Theorem \ref{thm1} above embraces the partial regularity results previously obtained in \cite[Theorem 1.1]{AE18}.

Note that this generalization of partial regularity result from \cite{AE18} to Theorem \ref{thm1} here is nontrivial. Indeed, due to the abstract form of $ \alpha(x,u) $ and $ \beta(x,u)$, implementing the strategy from \cite{AE18} becomes a delicate task. In addition, equation \eqref{intro:eq1} is no longer homogeneous caused by the presence of different and variable gradient power, which makes the scaling process more tricky. In contrast with the single power degeneracy case \cite{AE18, LS13}, the quantities involving the gradient variable on the left hand side of \eqref{intro:eq1} are not identically preserved after scaling. Furthermore, although such the operator $ \Delta_{p}^{\rm{N}} $ has a uniformly elliptic structure (see Remark \ref{rk6}), the strategy of \cite{LS13} cannot be applied directly to (\ref{intro:eq1}) since the discontinuity at the set $ \{Du=0\} $. This challenge necessitates the development of a new technique to address the problem effectively.

To overcome these difficulties, we adopt the idea developed in \cite{DJ22}. More precisely, we shall consider the existence and regularity of viscosity solutions to a Dirichlet problem associated with the anisotropic free transmission problems
\begin{align}
\label{eq6}
 \left\{
     \begin{aligned}
     & -\bigg\{ |Du|^{a_{1}\chi_{\{u>0\}}+a_{2}\chi_{\{u<0\}}}+a(x)\chi_{\{u>0\}}|Du|^{a_{1}}
     \\
      & \qquad   \qquad   \qquad  \qquad   \qquad   \qquad   +a(x)\chi_{\{u<0\}}|Du|^{a_{2}}\bigg\}\Delta_{p}^{\rm{N}}u=f(x) \qquad \ \text{in} \ \ B_{1}         \\
     & u=g  \qquad \qquad \qquad \qquad \qquad \qquad \qquad \qquad \qquad \qquad \qquad \qquad \qquad \ \   \text{on}  \  \partial B_{1}                  ,          \\
     \end{aligned}
     \right.
\end{align}
where $ a_{1}, a_{2} $ are nonnegative constants, $ 0\leq a(x) \in C(B_{1}) $, $f(x) \in C(B_{1})\cap L^{\infty}(B_{1}) $ and $ g \in C(\partial B_{1})  $. The existence of solutions to problem \eqref{eq6} can be accomplished via Perron's method and fixed point argument, see for instance \cite{De22, HPRS21, PS22}. Once this matter has been settled, the regularity for solution obtained above relies on a key observation, which is that the viscosity solution $ u\in C(B_{1}) $ of \eqref{eq6} turns out to be a viscosity sub-solution and viscosity super-solution to
\begin{align}\label{1.10}
\begin{split}
\min_{i=0,1,2} \big  \{-|Du|^{a_{i}}\Delta_{p}^{\rm{N}}u \big \}  = ||f||_{L^{\infty}(B_{1})}
\end{split}
\end{align}
and
\begin{align}\label{1.11}
\begin{split}
\max_{i=0,1,2} \big \{ -|Du|^{a_{i}}\Delta_{p}^{\rm{N}}u \big \}  =-||f||_{L^{\infty}(B_{1})},
\end{split}
\end{align}
respectively, where $ a_{0} = 0 $.
Under these two viscosity inequalities \eqref{1.10} and \eqref{1.11}, we combine the arguments established in \cite{De22, AE18} and recent work of \cite{HPRS21} to get $ C^{1,\alpha'} $ regularity of solution $ u $ obtained above.

To be more precise, our goal is to show the graph of $ u $ can be approximated by an affine function with an error bounded by $ Cr^{1+\alpha'} $ in any ball of radius $ r $. For this purpose, we first show that H\"{o}lder regularity of solution $ u $ for perturbed equation
\begin{align}\label{Intro:eqN}
\begin{split}
\min \bigg \{-\Delta_{p,\xi}^{\rm{N}}u, -|Du+\xi|^{a_{1}}\Delta_{p,\xi}^{\rm{N}}u,
-|Du+\xi|^{a_{2}} \Delta_{p,\xi}^{\rm{N}}u \bigg \} = ||f||_{L^{\infty}(B_{1})}
\end{split}
\end{align}
and
\begin{align}\label{Intro:eqM}
\begin{split}
\max \bigg \{-\Delta_{p,\xi}^{\rm{N}}u, -|Du+\xi|^{a_{1}}\Delta_{p,\xi}^{\rm{N}}u,  -|Du+\xi|^{a_{2}}\Delta_{p,\xi}^{\rm{N}}u \bigg \} = -||f||_{L^{\infty}(B_{1})},
\end{split}
\end{align}
where $ \xi $ is an arbitrary vector in $ \mathbb{R}^{n} $ and
\begin{equation*}
\Delta_{p,\xi}^{\rm{N}}u:= \Delta u + (p-2)\bigg\langle D^{2}u\frac{Du+\xi}{|Du+\xi|},\frac{Du+\xi}{|Du+\xi|}\bigg\rangle, \ \ 1<p<\infty .
\end{equation*}
This estimate can be carried out using the method introduced by Ishii and Lions \cite{HL90} for large/small scopes. Then, resorting to compactness result, which, together with the regularity of the limiting solutions (see Lemma \ref{lem4.21}) to show improvement of flatness for solutions of \eqref{Intro:eqN} and \eqref{Intro:eqM} in Lemma \ref{lem4.22}. Next we prove $ C^{1,\alpha'}$ regularity of solutions to \eqref{1.10} and \eqref{1.11} via iteration manner and Lemma \ref{lem4.22}. In the end, $ C^{1,\alpha'} $ regularity for solutions to \eqref{intro:eq1} is an outcome of Theorem \ref{pro2}, see Section \ref{Section 5}.

It can be seen from our Theorem \ref{thm1} and the previous results \cite{De21, YRZ21, AE18, AEGE21, LS13, ART15} that there is an intrinsic dependence between the obtained regularity and the rate of degeneracy, see Table \ref{Table1} below, for details.
\begin{table}[!htp]\label{Table1}
   \centering
   {
    \linespread{1.0}  \selectfont
    \caption{An intrinsic dependence between the obtained regularity and the rate of degeneracy. Here $ \alpha_{0} $ is the optimal H\"{o}lder exponent for solutions to $ F$ harmonic function.}
   \begin{tabular}{||c|c|c||}
        \hline
        Degenerate equation & Regularity of $ u $ & Reference  \\
        \hline
        \multirow{2}{*}{$ |Du|^{\gamma} F(D^{2}u) =f $} & $ C^{1,\alpha}_{loc}, \ 0< \alpha \leq \frac{1}{1+\gamma}$ &  \cite[Theorem 1]{LS13}  \\
        %\cline{2-3}
         & $ C^{1,\alpha}_{loc}, \ \alpha = \min \big\{ \alpha_{0}^{-}, \frac{1}{1+\gamma}   \big\}$ &  \cite[Theorem 3.1]{ART15}  \\
        \hline
        $ \big\{|Du|^{\gamma_{1}} + a(x) |Du|^{\gamma_{2}} \big\}F(D^{2}u) =f $ & $ C^{1,\alpha}_{loc}, \ 0< \alpha \leq \frac{1}{1+\gamma_{1}}$ & \cite[Theorem 1]{De21} \\
        \hline
        $ |Du|^{\gamma(x)} F(D^{2}u) = f $  & $ C^{1,\alpha}_{loc}, \ \alpha = \min \big\{ \alpha_{0}^{-}, \frac{1}{1+||\gamma||_{\infty}}   \big\} $ &  \cite[Theorem 1.2]{AEGE21}  \\
         \hline
         $ \big\{|Du|^{\gamma_{1}(x)} + a(x) |Du|^{\gamma_{2}(x)} \big\}F(D^{2}u) =f $ & $ C^{1,\alpha}_{loc}, \ \alpha = \min \big\{ \alpha_{0}^{-}, \frac{1}{1+||\gamma_{1}||_{\infty}}   \big\} $  & \cite[Theorem 1.1]{YRZ21} \\
         \hline
         $ |Du|^{\gamma} \Delta_{p}^{\rm{N}} u =f $ & $ C^{1,\alpha}_{loc}, \ \ 0< \alpha \leq \frac{1}{1+\gamma} $  &  \cite[Theorem 1.1]{AE18}  \\
         \hline
         $ -\big\{|Du|^{\alpha(x,u)}+a(x)|Du|^{\beta(x,u)}\big\}\Delta_{p}^{\rm{N}} u =f(x)$  &  $ C^{1,\alpha'}_{loc}, \ \ 0< \alpha' \leq \frac{1}{1+a_{2}}  $  &  Theorem \ref{thm1}  \\
         \hline
   \end{tabular}
   }
\end{table}

Therefore, if this rate is varied over the domain, it is natural to expect regularity results that vary over the domain as well. The first breakthrough in this direction came up in the work of Jesus in \cite{DJ22}, where he showed pointwise optimal $ C^{1,\alpha} $ regularity for degenerate fully nonlinear equation. More precisely, he showed the following (see Theorem 2 in \cite{DJ22}):

\begin{TheoremA}(\cite[Theorem 2]{DJ22})\label{extra:thm1}
Let $ u \in C(B_{1}) $ be a viscosity solution to
\begin{equation*}
  -|Du|^{\beta(x,u,Du)}F(D^{2}u)=f(x) \ \ \text{in} \ \ B_{1},
\end{equation*}
where $ \beta(x,u,Du)= \sum_{i=0}^{N}\beta_{i}(x)\chi_{G_{i}(u,Du)}(x) $, $ \{G_{i}(u,Du)\}_{i=0}^{N} $ are disjoint sets in $ B_{1}$, $ \{\beta_{i}(x)\}_{i=0}^{N} $ are uniformly bounded from above and below, $ F $ is a uniformly elliptic operator, and $ f \in C(B_{1}) \cap L^{\infty}(B_{1})$. Assume also $ \{\beta_{i}(x)\}_{i=0}^{N} $ have modulus of continuity $ \omega $ satisfying
\begin{equation*}
  \limsup_{t\rightarrow 0} \ln(t^{-1}) \omega(t) =0.
\end{equation*}
Then for every $ x_{0} \in B_{1/2}$, it holds $ u $ is $C^{1,\alpha}(x_{0}) $, and
\begin{equation*}
  ||u||_{C^{1,\alpha}(x_{0})} \leq C \big( ||u||_{L^{\infty}(B_{1})} + ||f||_{L^{\infty}(B_{1})}   \big),
\end{equation*}
where
\begin{equation*}
  \alpha = \min_{i=0,1,\cdots,N} \bigg\{\alpha_{0}^{-}, \frac{1}{1+\beta_{i}(x_{0})}     \bigg\}.
\end{equation*}
\end{TheoremA}

The model in Theorem \ref{extra:thm1} being studied is a diffusion process, which degenerates as a power of the gradient. The degeneracy law depends on the division of regions, which is discontinuous along $ G_{0}$.
This is a typical {\it transmission problem} describing the diffusion process within heterogeneous media. Applications include thermal and electromagnetic conductivity as well as composite materials, such as fiber-reinforced structures.
We refer the readers to \cite{B10} a more detailed exploration of this topic.

In the proof of Theorem \ref{extra:thm1}, the author adapted a strategy (the approximation method and iteration argument) from \cite{ART15, AEGE21, YRZ21, SSKS221}. Motivated by \cite{DJ22}, in this paper, we will establish a similar pointwise estimate as in Theorem \ref{extra:thm1} above for equation \eqref{intro:eq1}. It is noteworthy to mention that our approach is based on a new improved oscillation-type estimate combined with a localized analysis, which is totally different from the approach in Theorem \ref{extra:thm1}.

In order to state our second main results, we first make some essential assumptions.

Let $ \Omega_{i}(u) \subset B_{1}, i=1,2,\cdots,M $ be disjoint sets which depend on the solution $ u $, and define $ \Omega_{0}(u) := B_{1} \setminus \cup_{i=1}^{M}\Omega_{i} $. Assume also $ \alpha(x,u), \beta(x,u) $ have the form
\begin{equation}\label{intro:eq115}
\alpha(x,u)=\sum_{i=0}^{M}\alpha_{i}(x)\chi_{\Omega_{i}(u)},  \ \
\beta(x,u)=\sum_{i=0}^{M}\beta_{i}(x)\chi_{\Omega_{i}(u)}
\end{equation}
with
\begin{equation}\label{intro:eq116}
0<a_{1}\leq \alpha_{i}(x) \leq \beta_{i}(x)\leq a_{2} <+\infty,
\end{equation}
where $\chi_{\Omega_i}$ are characteristic functions of $\Omega_i $.
Moreover, we require additional assumptions on $ \alpha_{i}(x) $ and $ \beta_{i}(x),i=0,1,\cdots,M $.
Suppose that there is a non-decreasing function $ \omega: [0,+\infty ) \rightarrow  [0,+\infty ) $ such that
\begin{equation}\label{intro:eq117}
|\alpha_{i}(x)-\alpha_{i}(y)|+|\beta_{i}(x)-\beta_{i}(y)| \leq  \omega (|x-y|), \ i=0,1,\cdots,M
\end{equation}
and $ \omega$ satisfies the balancing condition
\begin{equation}\label{intro:eq118}
\limsup_{t\rightarrow 0} \omega(t) \ln(t^{-1}) =0.
\end{equation}

The second main result of this paper is the following:
\begin{Theorem}(Almost optimal pointwise $ C^{1,\tau} $ regularity)
\label{thm2}
Suppose $ u $ is a bounded viscosity solution of \eqref{intro:eq1} and the assumptions \eqref{intro:eq115}--\eqref{intro:eq118} hold. Then for every $x_{0} \in B_{1/2}  $, we have that $ u $ is $ C^{1,\tau}(x_{0}) $. More precisely, we have
\begin{equation}
\label{eq16}
||u||_{C^{1,\tau}(x_{0})} \leq C \left(1+||u||_{L^{\infty}(B_{1})}+||f||_{L^{\infty}(B_{1})}^{\gamma}\right),
\end{equation}
where
$$ \tau = \min_{i=0,1,\cdots,M} \bigg\{ \widehat{\beta_{0}}^{-}, \frac{1}{1+\alpha_{i}(x_{0})} \bigg\}, \ \ \gamma= \frac{1}{1+\min\limits_{i=0,1,\cdots,M}\inf_{B_{1}}\alpha_{i}(x)} ,$$
and $ C=C(n,p,a_{1},a_{2},\omega)$ and $ \widehat{\beta_{0}} \in (0,1) $ is the constant in the homogeneous case of Lemma \ref{lem2.3}.
\end{Theorem}

Note that the constant $ \widehat{\beta_{0}} $ in Theorem \ref{thm2} is the nearly optimal H\"{o}lder exponent to the gradient for solutions of the normalized $p$-harmonic function, see Section \ref{Section 2}.

\vspace{2mm}

Before proceeding further, we make the following important remarks.

\begin{remark}\label{rk3}
The condition \eqref{intro:eq118} admits an equivalent assertion in \cite[Section 2]{DJ22}: for any fixed $ \delta_{1} $ such that if $ \rho \leq \delta_{1} $, then for every $ k \in \mathbb{N} $,
$$ k\omega(\rho^{k}) \leq \frac{\widehat{\beta_{0}}-\tau}{2},$$
where $ \tau $ and $\widehat{\beta_{0}}$ are the constants in Theorem \ref{thm2}.
\end{remark}

\begin{remark}\label{rk4}
Let us consider variable exponents satisfying the Log-condition(see \cite[Section 5]{BDRR23}):
\begin{equation*}
|\alpha_{i}(x)-\alpha_{i}(y)| + |\beta_{i}(x)-\beta_{i}(y)| \leq \frac{\omega^{*}(|x-y|)}{|\log(|x-y|^{-1})|}, \ i=0,1,\cdots,M, \ \forall x, y \in B_{1}, \ x \neq y
\end{equation*}
for a universal modulus of continuity $ \omega^{*}:  [0,+\infty ) \rightarrow  [0,+\infty )$. Note that the function $ r \mapsto \frac{\omega^{*}(r)}{|\log(r^{-1})|}$ is nondecreasing on $ (0,r^{*})$ for some $r^*>0$ with $ \lim_{r\rightarrow 0} \frac{\omega^{*}(r)}{|\log(r^{-1})|} =0$.
Hence the assumption \eqref{intro:eq118} does hold. Particularly, such a condition plays a decisive role in proving higher regularity of solutions to equations with variable exponents (see \cite[Section 4]{MP21}).
\end{remark}

\vspace{2mm}

{\bf {Novelty of this paper.}}
Here, we will briefly explain the new features in this paper.

\vspace{2mm}

(1) The analysis in \cite[Theorem 1.1]{AE18} is limited to the case when $ a(x) =0 $ and $ \alpha(x, u)$ is a constant. Theorem \ref{thm1} generalizes the partial regularity result in \cite{AE18} by allowing $ 0 \leq a(x) \in C(B_{1})$. We only need to assume the positive lower and upper bounds of $ \alpha(x,u) $ and $\beta(x,u)$. No continuity requirements of $ \alpha(x,u) $ and $\beta(x,u)$ are needed. Moreover, this finding allows us to discover an interesting new proof based on two core viscosity inequalities \eqref{1.10} and \eqref{1.11}.

\vspace{2mm}

(2) Theorem \ref{thm2} establishes asymptotically optimal pointwise $ C^{1,\tau}$ regularity for problem \eqref{intro:eq1}, which gives an explicit characterization of $ \tau $ in terms of the degeneracy rates. Even for the simplest model when $ M=2 $, $ \Omega_{1}(u):= \{u>0\}  $, $ \Omega_{2}(u):= \{u<0\} $, $ \Omega_{0}(u):= \{u=0\} $ and $ \alpha_{i}, \beta_{i} $ are constants, $ i=0,1,2 $, there seems no result on the study of pointwise regularity of solutions for \eqref{intro:eq1}.

\vspace{2mm}

(3) Proving Theorem \ref{thm2} is the most delicate part of this paper. Unlike \cite[Theorem 3.1]{ART15}, \cite[Theorem 1.2]{AEGE21}, \cite[Theorem 1.1]{YRZ21}, \cite[Theorem 2]{DJ22} and \cite[Theorem 1.1]{SSKS221}, our approach does not make use of the compactness of solutions for perturbed equation, the approximation method and iteration argument. Our strategy is based on improved oscillation-type estimate (see \cite{AME17, JG20}) combined with a localized analysis (see \cite{ATU17, AEH15, DV21}). We believe that this alternate viewpoint would definitely be of independent interest.
\begin{enumerate}[-]
  \item (\underline{improved oscillation-type estimate}). We aim to perform a geometric decay argument along those points where the gradient becomes very small (in a suitable manner). For this purpose, we first show that the solutions $ u $ of \eqref{intro:eq1} can be approximated by normalized $ p $-harmonic function in a $ C^{1}_{loc} $ fashion. This approximation result ensures an interesting oscillation estimate for solutions to \eqref{intro:eq1} near the critical set $ \{x:Du(x)=0\} $, i.e,
      \begin{equation*}
        \sup_{B_{\rho}(x_{0})} |u(x)-u(x_{0})| \leq \rho^{1+\tau'} +|Du(x_{0})|\rho
      \end{equation*}
     for every $ 0< \tau' < \widehat{\beta_{0}} $ and $ \rho \in (0,\frac{1}{2}) $. Then by iterating the oscillation estimate above, we obtain
     \begin{equation}\label{Intro:eq120}
       \sup_{B_{\rho^{k}}(x_{0})} |u(x)-u(x_{0})| \leq \rho^{k(1+\tau_{k})} +|Du(x_{0})| \sum_{i=0}^{k-1} \rho^{k+i\tau_{k}}
     \end{equation}
   for a non-decreasing sequence $ \{ \tau_{k}\} $.

  \item (\underline{localized analysis}). Whenever $ |Du(x_{0})| $ is small, by using improved oscillation-type estimate in \eqref{Intro:eq120}, we obtain that $ u $ is $ C^{1,\tau}(x_{0}) $. On the other hand, whenever $ |Du(x_{0})| $ is large, then the equation becomes non-degenerate and classical estimate can apply.
\end{enumerate}

{\bf{State-of-the-art.}} In recent years, there has been increasing attention on equations driven by degenerate operators.

Regarding obstacle-type problems, Da Silva-Vivas in \cite{DV21} and \cite{2DV21} established existence and optimal regularity estimates for degenerate elliptic models in non-divergence form:
\begin{equation*}
   \min \{f-|\nabla u|^{\gamma} \Delta u, u -\phi \} = 0 \ \ \text{in} \ \ B_{1} \ \ \left(\text{resp.} \ \ |\nabla u|^{\gamma} = 1\chi_{\{u>\phi\}}\right),
\end{equation*}
where $ \gamma > 0 $, $ f \in L^{\infty}(B_{1}) $, and $ \phi \in C^{1,\beta}(B_{1}) $ with $ \beta \in (0,1] $. Their results show solutions belong to $ C^{1,\min\{\beta,\frac{1}{1+\gamma}\}} $, with \cite{2DV21} employing a geometric approach to derive sharp regularity estimates for $ |\nabla u|^{\gamma} = 1\chi_{\{u>\phi\}} $. Complementary to this, Da Silva et al. \cite{DLRR21} obtained geometric regularity estimates for dead-core problems governed by fully nonlinear elliptic operators of degenerate/singular type. These advances are contextualized by broader studies on $ p$-Laplacian tug-of-war games \cite{P24} and divergence-form quasilinear elliptic models \cite{MZ97}, collectively enriching the framework for analyzing degenerate normalized $ p$-Laplacian operators.

The field has seen extensive developments in elliptic and parabolic normalized $ p $-Laplacian equations, including degenerate/singular cases \cite{MPZ24, PM22, A20, CTL19, TL17, FZ23, LY23, DDS25}. Such as, in the elliptic case, Miao-Peng-Zhou \cite{MPZ24} established a nonlinear Calder\'{o}n-Zygmund $ L^{2}$-theory to the Dirichlet problem
\begin{equation*}
  - |\nabla u|^{\gamma} \Delta_{p}^{\rm{N}} u = f \in L^{2}(B_{1}) \ \ \text{in} \ \ B_{1};  \ \  u = 0 \ \ \text{on} \ \ \partial B_{1}
\end{equation*}
for $ n \geq 2, p>1$ and a large range of $ \gamma > -1 $. Very recently, in the parabolic setting, Bessa-Da Silva-S\'{a} in \cite{DDS25} showed the existence and sharp geometric regularity estimates for bounded solutions of a class of quasilinear parabolic equations in non-divergence form with non-homogeneous degeneracy:
\begin{equation*}
  \partial_{t} u = (|\nabla u|^{\mathcal{P}}+ \mathcal{A}(x,t)|\nabla u|^{\mathcal{Q}}) \Delta_{p}^{\rm{N}} u + f(x,t)  \ \ \text{in} \ \ Q_{1} = B_{1} \times (-1,0],
\end{equation*}
where $ p \in (1,\infty) $, $ \mathcal{P},\mathcal{Q} \in [0, \infty)$, and the functions $ \mathcal{A}, f: Q_{1} \rightarrow \mathbb{R} $ are suitably defined. In such a content, their approach is based on geometric tangential methods, combining a refined oscillation mechanism, compactness arguments, and scaling techniques.

Two key directions emerge for further investigation: first, extending Theorem \ref{thm1} to scenarios with $ p >2$, $ a(x)=0$, $ \alpha(x,u):=a_{1}>0$, and $ f \in C(B_{1}) \cap L^{q'}(B_{1})$ ($ q' >0$); second, exploring $ C^{1,\alpha} $ regularity for Neumann/oblique boundary problems of degenerate/singular normalized $ p $-Laplacian equations \textemdash~both currently unexplored and slated for future work.

{\bf{Organization of the paper.}} In Section \ref{Section 2}, we introduce some definitions of viscosity solutions and collect two useful lemmas that will be needed in the proofs of the main theorems. The existence of solution of \eqref{eq6} is given in Section \ref{Section 3}. Section \ref{Section 4} is dedicated to the proof of local $ C^{1,\alpha'}$ regularity to \eqref{eq6}. We complete the proof of Theorem \ref{thm1} in Section \ref{Section 5}. In Section \ref{Section 6}, we are devoted to presenting the proof of Theorem \ref{thm2}.

{\bf{Notation.}} The following notations are also used in this article.

\begin{itemize}
\item For $ r >0 $, $ B_{r}(x) $ denotes the open ball of radius $ r $ centered at $ x $. We simply use $ B_{r} $ to denote the open ball $ B_{r}(0)$.

\item $ \textbf{I}\rm{d}_{n} $ denotes the $n\times n$ identity matrix.

\item $ \chi_{\rm{E}} $ denotes the characteristic function of measurable set $ E $.

\item For $ \alpha \in (0,1) $, we shall write
$$ ||u||_{C^{0,\alpha}(B)}:= ||u||_{L^{\infty}(B)} + \sup_{x,y\in B}\frac{|u(x)-u(y)|}{|x-y|^{\alpha}} $$
 and
 $$ ||u||_{C^{1,\alpha}(B)}:=\sup _{r>0,x\in B} \inf_{p\in \mathbb{R}^{n},c\in \mathbb{R} } \sup_{z\in B_{r}(x)\cap B} \frac{|u(z)-p\cdot z-c|}{r^{1+\alpha}} , $$
where $ B \subset B_{1} $ is an open ball.

\item We say $ u $ is $ C^{1,\alpha} $ at $ x_{0} \in B_{1} $ if $ u \in C^{1} $ in a neighborhood of $ x_{0} $ and
$$ ||u||_{C^{1,\alpha}(x_{0})} := \sup_{y \in B_{r}(x_{0})}|u(y)|+\sup_{y \in B_{r}(x_{0})}|Du(y)|+\sup_{y\in B_{r}(x_{0}),r>0}\frac{|Du(y)-Du(x_{0})|}{|y-x_{0}|^{\alpha}}<\infty . $$

\item $ C $ shall denote a generic positive constant which may vary in different inequalities.
\end{itemize}

\vspace{3mm}

\section{Preliminaries}\label{Section 2}

We shall split this section into two parts: First, we give the definitions of the viscosity solution for the normalized $ p $-Laplacian equation and equation $ (\ref{intro:eq1})$. Then, we collect several useful lemmas, which will be used in the proofs of the main theorems.

\subsection{Definitions of viscosity solutions}
We define the viscosity solutions of the normalized $ p$-Laplacian equation and the equation \eqref{intro:eq1}.

\begin{Definition}(\cite[Section 2]{CIL92})
\label{def1}
Let $ 1<p< \infty $. An upper semicontinuous function $ u $ is a viscosity sub-solution of the equation $ -\Delta_{p}^{\rm{N}}u =f  $ if for all $ x_{0} \in B_{1}  $ and $ \varphi \in C^{2}(B_{1}) $ such that $ u- \varphi $ attains a local maximum at $ x_{0}$, one has
\begin{equation*}
\left\{
     \begin{aligned}
     & -\Delta_{p}^{\rm{N}}\varphi(x_{0}) \leq f(x_{0}) ,  \qquad   \qquad  \qquad  \qquad \qquad  \qquad \text{if} \ D \varphi(x_{0}) \neq 0,     \\
     &  -\Delta \varphi(x_{0}) - (p-2) \lambda_{max}(D^{2}\varphi(x_{0}) ) \leq f(x_{0}), \ \ \ \ \ \ \text{if} \ D \varphi(x_{0}) = 0 \ \text{and} \  p\geq 2 ,    \\
     & -\Delta \varphi(x_{0}) - (p-2) \lambda_{min}(D^{2}\varphi(x_{0}) ) \leq f(x_{0}), \ \ \ \ \ \ \text{if} \ D \varphi(x_{0}) = 0 \ \text{and} \  1< p < 2 .     \\
     \end{aligned}
     \right.
\end{equation*}
A lower semicontinuous function $ u $ is a viscosity super-solution of the equation $ -\Delta_{p}^{\rm{N}}u =f  $ if for all $ x_{0} \in B_{1}  $ and $ \varphi \in C^{2}(B_{1}) $ such that $ u- \varphi $ attains a local minimum at $ x_{0}$, one has
\begin{equation*}
\left\{
     \begin{aligned}
     & -\Delta_{p}^{\rm{N}}\varphi(x_{0})  \geq f(x_{0}) ,  \qquad   \qquad  \qquad  \qquad \qquad  \qquad \text{if} \ D \varphi(x_{0}) \neq 0,     \\
     &  -\Delta \varphi(x_{0}) - (p-2) \lambda_{min}(D^{2}\varphi(x_{0}) ) \leq f(x_{0}), \ \ \ \ \ \ \text{if} \ D \varphi(x_{0}) = 0 \ \text{and} \  p\geq 2 ,    \\
     & -\Delta \varphi(x_{0}) - (p-2) \lambda_{max}(D^{2}\varphi(x_{0}) ) \leq f(x_{0}), \ \ \ \ \ \ \text{if} \ D \varphi(x_{0}) = 0 \ \text{and} \  1< p < 2 .     \\
     \end{aligned}
     \right.
\end{equation*}

We say that $ u $ is a viscosity solution of $ -\Delta_{p}^{\rm{N}}u =f  $ in $ B_{1} $ if it is both a viscosity sub-solution and a viscosity super-solution.
\end{Definition}

For $ 0 <a_{1} \leq \alpha(x,u) \leq \beta(x,u)\leq a_{2} <\infty $, the operator $$ -\big\{|Du|^{\alpha(x,u)}+a(x)|Du|^{\beta(x,u)}\big\}\Delta_{p}^{\rm{N}} u $$ is well-defined (in the sense that it is not singular) even if $ Du=0 $. We then use the standard definition of viscosity solution in \cite{HL90} to define the viscosity solution of equation \eqref{intro:eq1}.

\begin{Definition}(\cite[Section 2]{HL90})
\label{def1}
An upper semicontinuous function $ u $ is a viscosity sub-solution of \eqref{intro:eq1} in $ B_{1} $ if for all $ \varphi \in C^{2}(B_{1}) $ such that $ u-\varphi $ has a local maximum at $ x_{0}\in B_{1} $ and
\begin{equation*}
-\big\{|D\varphi(x_{0})|^{\alpha(x_{0}, u(x_{0}))}+a(x_{0})|D\varphi(x_{0})|^{\beta(x_{0},u(x_{0}))}\big\}\Delta_{p}^{\rm{N}} \varphi(x_{0})  \leq f(x_{0}).
\end{equation*}
A lower semicontinuous function $ u $ is a viscosity super-solution of (\ref{intro:eq1}) in $ B_{1} $ if for all $ \varphi \in C^{2}(B_{1}) $ such that $ u-\varphi $ has a local minimum at $ x_{0}\in B_{1} $ and
\begin{equation*}
-\big\{|D\varphi(x_{0})|^{\alpha(x_{0},u(x_{0}))}+a(x_{0})|D\varphi(x_{0})|^{\beta(x_{0},u(x_{0}))}\big\}\Delta_{p}^{\rm{N}} \varphi(x_{0})  \geq f(x_{0}).
\end{equation*}

We say that $ u $ is a viscosity solution of \eqref{intro:eq1} in $ B_{1} $ if it is both a viscosity sub-solution and a viscosity super-solution.
\end{Definition}

\begin{remark}
When $ \alpha(x,u), \beta(x,u) \in C^{0}(B_{1})   $ and $ -1 < \alpha_{\min}\leq\alpha(x,u) \leq \beta(x,u) \leq \beta_{\max} < 0    $, the definition of viscosity solution to $ ( \ref{intro:eq1} )$ can be adapted from the definition used by Birindelli-Demengel in \cite{IF10, IF06, IF04}.
A lower semicontinuous function $ u: \overline{B}_{1} \rightarrow \mathbb{R} $ is called a viscosity super-solution of $ ( \ref{intro:eq1} )$ if for each $ x_{0} \in B_{1}  $ either there exists $ \delta >0 $ such that $ u $ is constant in $ B(x_{0},\delta) $ and $ f(x) \leq 0 $ for all $ x \in B(x_{0},\delta) $ or for all $ \varphi \in C^{2}(B_{1}) $ such that $ u- \varphi$ has a local minimum at $ x_{0} $ and $ D \varphi(x_{0}) \neq 0   $, it holds true that
\begin{equation*}
-\big\{|D\varphi(x_{0})|^{\alpha(x_{0},u(x_{0}))}+a(x_{0})|D\varphi(x_{0})|^{\beta(x_{0},u(x_{0}))}\big\}\Delta_{p}^{\rm{N}} \varphi(x_{0})  \geq f(x_{0}).
\end{equation*}
A viscosity sub-solution of $ ( \ref{intro:eq1} )$ can be defined analogously. A function $ u$ is called a viscosity solution to $ ( \ref{intro:eq1} )$ if and only if it is both a viscosity super-solution and a viscosity sub-solution.
\end{remark}

\begin{remark}\label{rk6}
We say the normalized $ p $-Laplacian operator is uniformly elliptic in the sense that
\begin{equation*}
\label{eq0}
\mathcal{M}^{-}_{\lambda,\Lambda}(D^{2}u) \leq \Delta_{p}^{\rm{N}}u \leq \mathcal{M}^{+}_{\lambda,\Lambda}(D^{2}u),
\end{equation*}
where
$$ \mathcal{M}^{-}_{\lambda,\Lambda}(D^{2}u) £º= \inf_{A \in A_{\lambda,\Lambda}} \text{Tr}(AD^{2}u), \  \mathcal{M}^{+}_{\lambda,\Lambda}(D^{2}u) £º= \sup_{A \in A_{\lambda,\Lambda}} \text{Tr}(AD^{2}u)    $$
and $ A_{\lambda,\Lambda} $ is a set of symmetric $ n \times n $ matrices, whose eigenvalues belong to the interval $ [\lambda,\Lambda] $. Indeed, the normalized $ p $-Laplacian operator can be written in the form
$$ \Delta_{p}^{\rm{N}}u = \text{Tr}\left((\textbf{I}d_{n}+(p-2)\frac{Du}{|Du|}\otimes \frac{Du}{|Du|} )D^{2}u\right) ,  $$
then it is easy to check that $ \lambda =\min \big \{ 1,p-1  \big \}  $ and $ \Lambda =\max \big \{ 1,p-1  \big \} $.
\end{remark}

\subsection{Auxiliary lemmas}
We recall two lemmas related to the normalized $ p$-Laplacian, which will be used in later sections.

The first lemma is a useful variant of the cutting lemma, which is the key ingredient in the proof of Proposition \ref{prop6.2}.
\begin{Lemma}\label{lem2.2}
Assume that $ u $ is a viscosity sub-solution to
$$ \min_{i=0,1,\cdots,M} \bigg \{-\big(|Du|^{\alpha_{i}(x)} + a(x)|Du|^{\beta_{i}(x)} \big)   \Delta_{p}^{\rm{N}} u \bigg\}  = 0 \ \ \text{in} \ \ B_{1}      $$
and $ u $ is a viscosity super-solution to
$$ \max_{i=0,1,\cdots,M} \bigg \{-\big(|Du|^{\alpha_{i}(x)} + a(x)|Du|^{\beta_{i}(x)} \big)\Delta_{p}^{\rm{N}} u \bigg\} = 0 \ \ \text{in} \ \ B_{1},  $$
where $ \alpha_{i}(x), \beta_{i}(x) $ are given in \eqref{intro:eq115}--\eqref{intro:eq118}, $ i=0,1,\cdots,M $. Then $ u $ fulfills
$$ - \Delta_{p}^{\rm{N}}u =0 \ \ \text{in} \ \ B_{1}  $$ in the viscosity sense.
\end{Lemma}
The proof of this lemma closely parallels that of \cite[Lemma 5]{DJ22} or \cite[Lemma 2.6]{AE18}. We shall omit the proof.

In the end, we present an important lemma involving the nearly optimal regularity for solutions to the normalized $ p$-Poisson equation, which shall be used in the proof of Theorem \ref{thm2}. We refer the readers to \cite[Theorem 1.3]{AME17} for the details.
\begin{Lemma}\label{lem2.3}
Fixing an arbitrary constant $ \xi \in (0,\widehat{\alpha_{1}})$, where $ \widehat{\alpha_{1}} $ is the optimal H\"{o}lder exponent for gradients of $ p$-harmonic functions in terms of a priori estimate. Then the following conclusions hold:
\begin{itemize}
    \item[{\it (a).}] If $ p>1$, and $ f \in C(B_{1}) \cap L^{\infty}(B_{1}) $, then viscosity solutions to $ -\Delta_{p}^{\rm{N}}u=f $ are in $ C^{1,\widehat{\beta_{0}}}_{loc}(B_{1}) $, where $ \widehat{\beta_{0}}= \widehat{\alpha_{1}}-\xi   $;

    \item[{\it (b).}] If $ p>2 $, $ q > \max (2,n,p/2) $ and $ f \in C(B_{1}) \cap L^{q}(B_{1}) $, then viscosity solutions to $ -\Delta_{p}^{\rm{N}}u=f $ are in $ C^{1,\widehat{\beta_{0}}}_{loc}(B_{1}) $, where $ \widehat{\beta_{0}} := \min \{\widehat{\alpha_{1}}-\xi, 1-n/q\} $.
\end{itemize}
In particular, if the equation is homogeneous, namely $ f=0 $, then viscosity solutions to $ -\Delta_{p}^{\rm{N}}u=0 $ are in $ C^{1,\widehat{\beta_{0}}}_{loc}(B_{1}) $, where $ \widehat{\beta_{0}}= \widehat{\alpha_{1}} -\xi   $.
\end{Lemma}

\vspace{3mm}

\section{The existence of problem \eqref{eq6} }\label{Section 3}
	~~~~In this section, we show the existence of viscosity solution to problem \eqref{eq6}, which is formulated in the following theorem:

\begin{Theorem}
\label{pro1}
Assume the conditions $ (\ref{pre:eq1})-(\ref{pre:eq4})$ hold, and normalized $ p$-Laplacian operator $ \Delta_{p}^{\rm{N}}$ is given in \eqref{introd:eq1}. Then there exists a viscosity solution $ u \in C(B_{1}) $ to \eqref{eq6}, and $ u $ is a viscosity sub-solution of
\begin{equation}\label{Section3:eq1}
\min_{i=0,1,2} \bigg \{-\big(|Du|^{a_{i}}+a(x)|Du|^{a_{i}}\big)\Delta_{p}^{\rm{N}}u \bigg \}=||f||_{L^{\infty}(B_{1})},
\end{equation}
and is a viscosity super-solution of
\begin{equation}\label{Section3:eq2}
\max_{i=0,1,2} \bigg \{-\big(|Du|^{a_{i}}+a(x)|Du|^{a_{i}}\big)\Delta_{p}^{\rm{N}}u \bigg \}=-||f||_{L^{\infty}(B_{1})},
\end{equation}
where $ a_{0} = 0 $.
\end{Theorem}

The proof is based on the approach in \cite[Section 3]{De22},  and \cite[Theorem 1]{HPRS21}, see also \cite[Theorem 1]{PS22}. First, we introduce a family of approximating problems and establish a comparison principle. Then, by the standard construction of a viscosity super-solution and a viscosity sub-solution, we can prove the existence and uniqueness of viscosity solution to the family of approximating problems via Perron's method. Using Schauder's fixed point theorem and a limiting procedure, we obtain a viscosity solution to problem \eqref{eq6}.
We shall present these steps in two subsections.

\subsection{Approximating problems and comparison principle}

First, we consider $ v \in C(\overline{B_{1}})$ such that $ v=g $ on $ \partial B_{1}$ and $ h_{\epsilon}^{v}(x) $ is defined by
\begin{equation*}
h_{\epsilon}^{v} = \zeta_{\epsilon}^{v} \ast g_{\epsilon}^{v},
\end{equation*}
where
\begin{equation*}
g_{\epsilon}^{v}=\max \left(\min (\frac{v+\epsilon}{2\epsilon},1), 0 \right) \ \ \ \text{in} \ \ B_{1}, \ g_{\epsilon}^{v}=0 \ \ \text{in} \ \ \mathbb{R}^{n} \setminus B_{1},
\end{equation*}
and $ \zeta_{\epsilon}^{v} $ is the standard mollifier function with $\epsilon>0$. Setting
$$  a_{\epsilon}^{v}(x)= a_{1}h_{\epsilon}^{v}(x) +a_{2}(1-h_{\epsilon}^{v}(x)),  $$
it is obvious that $$ a_{1} \leq a_{\epsilon}^{v}(x) \leq a_{2}    .$$ We consider a family of approximating problems
\begin{equation}\label{exi:eq1}
F_{\epsilon}^{v}(|Du|)\left(\epsilon u -\Delta u -(p-2) \bigg \langle D^{2}u \frac{Du}{\epsilon+|Du|},\frac{Du}{\epsilon+|Du|}\bigg\rangle \right)=f(x) \ \ \text{in} \ \ B_{1},
\end{equation}
where
\begin{equation*}
\begin{aligned}
F_{\epsilon}^{v}(|Du|) : = &(\epsilon+|Du|)^{a_{\epsilon}^{v}(x)}+ \big[\epsilon+a(x)h_{\epsilon}^{v}(x)\big](\epsilon+|Du|)^{a_{1}} \\
&   + \big[\epsilon+a(x)(1-h_{\epsilon}^{v}(x))\big](\epsilon+|Du|)^{a_{2}}.
\end{aligned}
\end{equation*}

The next lemma is an important comparison principle of problem \eqref{exi:eq1}.

\begin{Lemma}\label{lemma3.2}
Suppose $ 0 \leq a(x) \in C(B_{1}) $ and $ f(x) \in C(B_{1}) $. Let $ u $ be an upper semi-continuous viscosity sub-solution to \eqref{exi:eq1} and $ w $ be a lower semi-continuous viscosity super-solution to \eqref{exi:eq1}. If $ u \leq w $ on $ \partial B_{1}$, then $ u \leq w $ in $ B_{1} $.
\end{Lemma}
\begin{proof}
We prove by contradiction. Suppose $ \max_{B_{1}} (u-w) = a_{0} >0   $. For $ \delta >0 $, we define
$$ \varphi_{\delta}(x,y)= u(x)-w(y)-\frac{|x-y|^{2}}{2\delta}.$$
If there exists $ (x_{\delta},y_{\delta}) \in \overline{B_{1}} \times \overline{B_{1}} $ such that
\begin{equation}\label{exi:eq2}
\max_{(x,y)\in \overline{B_{1}} \times \overline{B_{1}}} \varphi_{\delta}(x,y) = \varphi_{\delta}(x_{\delta},y_{\delta}) \geq a_{0},
\end{equation}
from \cite[Lemma 3.1]{CIL92}, we have
\begin{equation}\label{exi:eq3}
\lim_{\delta \rightarrow 0}\frac{|x_{\delta}-y_{\delta}|^{2}}{\delta} =0.
\end{equation}
Notice that $ x_{\delta},y_{\delta}$ must belong to the interior of $ B_{1} $. Otherwise $ \varphi_{\delta}(x_{\delta},y_{\delta}) \leq 0 $, which is a contradiction to \eqref{exi:eq2}. By Ishii-Lions Lemma \cite[Theorem 3.2]{CIL92}, the limiting subjet $ (\frac{x_{\delta}-y_{\delta}}{\delta},X)$ of $ u $ at $ x_{\delta}$ and the limiting superjet $ (\frac{x_{\delta}-y_{\delta}}{\delta},Y)$ of $ u $ at $ y_{\delta} $ exist, which satisfy
\begin{equation}\label{exi:eq4}
-\frac{1}{\delta} \textbf{I}\rm{d}_{n} \leq
\begin{pmatrix}
X  &   0   \\
0  &    -Y
\end{pmatrix}
\leq \frac{1}{\delta}
\begin{pmatrix}
\textbf{I}\rm{d}_{n}   &   -\textbf{I}\rm{d}_{n}   \\
-\textbf{I}\rm{d}_{n} &   \textbf{I}\rm{d}_{n}
\end{pmatrix}.
\end{equation}
Moreover, we have the following two viscosity inequalities
\begin{equation}\label{exi:eq5}
H_{\epsilon;x_{\delta}}^{v}\left(\frac{|x_{\delta}-y_{\delta}|}{\delta}\right) \left(\epsilon u(x_{\delta}) -\Delta u(x_{\delta}) -(p-2) \bigg \langle X \frac{\frac{x_{\delta}-y_{\delta}}{\delta}}{\epsilon+|\frac{x_{\delta}-y_{\delta}}{\delta}|},\frac{\frac{x_{\delta}-y_{\delta}}{\delta}}{\epsilon+|\frac{x_{\delta}-y_{\delta}}{\delta}|}\bigg\rangle \right)\leq f(x_{\delta}),
\end{equation}
and
\begin{equation}\label{exi:eq6}
H_{\epsilon;y_{\delta}}^{v}\left(\frac{|x_{\delta}-y_{\delta}|}{\delta}\right) \left(\epsilon w(y_{\delta}) -\Delta w(y_{\delta}) -(p-2) \bigg \langle Y \frac{\frac{x_{\delta}-y_{\delta}}{\delta}}{\epsilon+|\frac{x_{\delta}-y_{\delta}}{\delta}|},\frac{\frac{x_{\delta}-y_{\delta}}{\delta}}{\epsilon+|\frac{x_{\delta}-y_{\delta}}{\delta}|}\bigg\rangle \right) \geq f(y_{\delta}),
\end{equation}
where
\begin{align}\label{exi:eq7}
\begin{split}
 H_{\epsilon;x_{\delta}}^{v}\left(\frac{|x_{\delta}-y_{\delta}|}{\delta}\right) := &\left(\epsilon+\frac{|x_{\delta}-y_{\delta}|}{\delta}\right)^{a_{\epsilon}^{v}(x_{\delta})}  + \bigg[\epsilon+ a(x_{\delta})h_{\epsilon}^{v}(x_{\delta})\bigg]\left(\epsilon+\frac{|x_{\delta}-y_{\delta}|}{\delta}\right)^{a_{1}} \\
 & +\bigg[ \epsilon+ a(x_{\delta})(1-h_{\epsilon}^{v}(x_{\delta}))\bigg]\left(\epsilon+\frac{|x_{\delta}-y_{\delta}|}{\delta}\right)^{a_{2}},
 \end{split}
\end{align}
and
\begin{align}\label{exi:eq8}
\begin{split}
 H_{\epsilon;y_{\delta}}^{v}\left(\frac{|x_{\delta}-y_{\delta}|}{\delta}\right) := &\left(\epsilon+\frac{|x_{\delta}-y_{\delta}|}{\delta}\right)^{a_{\epsilon}^{v}(y_{\delta})}  + \bigg[\epsilon+ a(y_{\delta})h_{\epsilon}^{v}(y_{\delta})\bigg]\left(\epsilon+\frac{|x_{\delta}-y_{\delta}|}{\delta}\right)^{a_{1}} \\
 & + \bigg[\epsilon+a(y_{\delta})(1-h_{\epsilon}^{v}(y_{\delta}))\bigg]\left(\epsilon+\frac{|x_{\delta}-y_{\delta}|}{\delta}\right)^{a_{2}}.
 \end{split}
\end{align}
Applying the matrix inequality \eqref{exi:eq4} to the vector $ ( \xi, \xi) \in \mathbb {R}^{2n} $, then we can readily derive
\begin{equation}\label{Y ge X}
 Y \geq X.
\end{equation}
For convenience, we denote
\begin{align}\label{exi:eq9}
\begin{split}
 -\Delta_{p;X}^{\rm{N}}u := -\Delta u(x_{\delta}) -(p-2) \bigg \langle X \frac{\frac{x_{\delta}-y_{\delta}}{\delta}}{\epsilon+|\frac{x_{\delta}-y_{\delta}}{\delta}|},\frac{\frac{x_{\delta}-y_{\delta}}{\delta}}{\epsilon+|\frac{x_{\delta}-y_{\delta}}{\delta}|}\bigg\rangle,
 \end{split}
\end{align}
and
\begin{align}\label{exi:eq10}
\begin{split}
 -\Delta_{p;Y}^{\rm{N}}w := -\Delta w(y_{\delta}) -(p-2) \bigg \langle Y \frac{\frac{x_{\delta}-y_{\delta}}{\delta}}{\epsilon+|\frac{x_{\delta}-y_{\delta}}{\delta}|},\frac{\frac{x_{\delta}-y_{\delta}}{\delta}}{\epsilon+|\frac{x_{\delta}-y_{\delta}}{\delta}|}\bigg\rangle.
 \end{split}
\end{align}
Now we combine \eqref{exi:eq5}--\eqref{exi:eq10} to infer that
\begin{align*}
\frac{f(x_{\delta})}{ H_{\epsilon;x_{\delta}}^{v}\left(\frac{|x_{\delta}-y_{\delta}|}{\delta}\right)}-\frac{f(y_{\delta})}{ H_{\epsilon;y_{\delta}}^{v}\left(\frac{|x_{\delta}-y_{\delta}|}{\delta}\right)} &- \epsilon (u(x_{\delta})-w(y_{\delta})) = \Delta_{p;Y}^{\rm{N}}w - \Delta_{p;X}^{\rm{N}}u   \\
 & = \text{Tr}(A(Y-X)) {\geq} 0,
\end{align*}
where
$$  A:= \textbf{I}\rm{d}_{n}+(p-2)\frac{\frac{x_{\delta}-y_{\delta}}{\delta}}{\epsilon+|\frac{x_{\delta}-y_{\delta}}{\delta}|}\otimes \frac{\frac{x_{\delta}-y_{\delta}}{\delta}}{\epsilon+|\frac{x_{\delta}-y_{\delta}}{\delta}|} >0.   $$
This, together with \eqref{exi:eq2}, yields that
\begin{align}\label{exi:eq11}
\begin{split}
0<\epsilon a_{0} &\leq \epsilon (u(x_{\delta})-w(y_{\delta})) \leq  \frac{f(x_{\delta})}{ H_{\epsilon;x_{\delta}}^{v}\left(\frac{|x_{\delta}-y_{\delta}|}{\delta}\right)}-\frac{f(y_{\delta})}{ H_{\epsilon;y_{\delta}}^{v}\left(\frac{|x_{\delta}-y_{\delta}|}{\delta}\right)}    \\
& \leq  \underbrace{\frac{f(x_{\delta})-f(y_{\delta})}{ H_{\epsilon;x_{\delta}}^{v}\left(\frac{|x_{\delta}-y_{\delta}|}{\delta}\right)}}_{:=D_{1}} +\underbrace{f(y_{\delta}) \bigg ( \frac{1}{H_{\epsilon;x_{\delta}}^{v}\left(\frac{|x_{\delta}-y_{\delta}|}{\delta}\right)} -\frac{1}{H_{\epsilon;y_{\delta}}^{v}\left(\frac{|x_{\delta}-y_{\delta}|}{\delta}\right)} \bigg )}_{:=D_{2}}.
\end{split}
\end{align}

We then estimate  the upper bounds of $D_1$ and $D_2$, respectively. In view of the continuity of $ f $ and \eqref{exi:eq7}, we deduce that
\begin{equation}\label{exi:eq12}
D_{1} \leq \epsilon ^{-\theta_{2}} \omega(|x_{\delta}-y_{\delta}|),
\end{equation}
where $ \omega $ denotes the modulus of continuity of $ f$. A direct computation yields that
\begin{align}\label{exi:eq13}
\begin{split}
D_{2} & \leq ||f||_{L^{\infty}(B_{1})} \frac{H_{\epsilon;y_{\delta}}^{v}\left( \frac{|x_{\delta}-y_{\delta}|}{\delta}   \right)-H_{\epsilon;x_{\delta}}^{v}\left( \frac{|x_{\delta}-y_{\delta}|}{\delta}   \right)}{H_{\epsilon;x_{\delta}}^{v}\left( \frac{|x_{\delta}-y_{\delta}|}{\delta}   \right)H_{\epsilon;y_{\delta}}^{v}\left( \frac{|x_{\delta}-y_{\delta}|}{\delta}   \right)}  \\
& = ||f||_{L^{\infty}(B_{1})} (A_{1} + A_{2} + A_{3}),
\end{split}
\end{align}
where
\begin{equation*}
\left\{
     \begin{aligned}
     & A_{1}:= \frac{\left(\epsilon+\frac{|x_{\delta}-y_{\delta}|}{\delta}\right)^{a_{\epsilon}^{v}(x_{\delta})}-\left(\epsilon+\frac{|x_{\delta}-y_{\delta}|}{\delta}\right)^{a_{\epsilon}^{v}(y_{\delta})}}{H_{\epsilon;x_{\delta}}^{v}\left( \frac{|x_{\delta}-y_{\delta}|}{\delta}   \right)H_{\epsilon;y_{\delta}}^{v}\left( \frac{|x_{\delta}-y_{\delta}|}{\delta}   \right)}   ;      \\
     & A_{2}:= \frac{\left(\epsilon+\frac{|x_{\delta}-y_{\delta}|}{\delta}\right)^{a_{1}}\bigg[a(x_{\delta})h_{\epsilon}^{v}(x_{\delta})-a(y_{\delta})h_{\epsilon}^{v}(y_{\delta})\bigg]}{H_{\epsilon;x_{\delta}}^{v}\left( \frac{|x_{\delta}-y_{\delta}|}{\delta}   \right)H_{\epsilon;y_{\delta}}^{v}\left( \frac{|x_{\delta}-y_{\delta}|}{\delta}   \right)}   ;      \\
     &  A_{3}:= \frac{\left(\epsilon+\frac{|x_{\delta}-y_{\delta}|}{\delta}\right)^{a_{2}}\bigg[a(x_{\delta})(1-h_{\epsilon}^{v}(x_{\delta}))-a(y_{\delta})(1-h_{\epsilon}^{v}(y_{\delta}))\bigg]}{H_{\epsilon;x_{\delta}}^{v}\left( \frac{|x_{\delta}-y_{\delta}|}{\delta}   \right)H_{\epsilon;y_{\delta}}^{v}\left( \frac{|x_{\delta}-y_{\delta}|}{\delta}   \right)}   .    \\
     \end{aligned}
     \right.
\end{equation*}
We next estimate the upper bounds of $A_1$, $A_2$ and $A_3$ respectively. Using the mean value theorem and the uniform boundedness of $ a_{\epsilon}^{v} $, we have
\begin{align}\label{exi:eq14}
\begin{split}
A_{1} & \leq \epsilon^{-2\theta_{2}} \bigg| \exp\left( -a_{\epsilon}^{v}(x_{\delta})\ln(\epsilon+\frac{|x_{\delta}-y_{\delta}|}{\delta})\right)- \exp\left( -a_{\epsilon}^{v}(y_{\delta})\ln(\epsilon+\frac{|x_{\delta}-y_{\delta}|}{\delta})\right)\bigg|   \\
 & \leq \exp(-a_{2}\ln^{\epsilon}) \big|a_{\epsilon}^{v}(x_{\delta})-a_{\epsilon}^{v}(y_{\delta})\big| \big|\ln^{(\epsilon+\frac{|x_{\delta}-y_{\delta}|}{\delta})}\big|   \\
 & \leq \epsilon^{-a_{2}}|\omega_{1}(|x_{\delta}-y_{\delta}|)|\left(|\ln^{\epsilon}| +\frac{|x_{\delta}-y_{\delta}|}{\delta}   \right),
\end{split}
\end{align}
where $ \omega_{1} $ denotes the modulus of continuity of $ a_{\epsilon}^{v} $.
A direct computation yields that
\begin{align}\label{exi:eq15}
\begin{split}
A_{2} & \leq \epsilon^{-2-a_{1}} \bigg[ h_{\epsilon}^{v}(x_{\delta})(a(x_{\delta})-a(y_{\delta})) + a(y_{\delta})(h_{\epsilon}^{v}(x_{\delta})-h_{\epsilon}^{v}(y_{\delta}))\bigg]  \\
& \leq C \epsilon^{-2-a_{1}} \bigg[ \omega_{2}(|x_{\delta}-y_{\delta}|) + \omega_{3}(|x_{\delta}-y_{\delta}|)\bigg],
\end{split}
\end{align}
where $ \omega_{2}$ and $\omega_{3}$ denote the modulus of continuity of $ a(x) $ and $ h_{\epsilon}^{v}(x) $, respectively.
Similar to $ A_{2} $, we can get
\begin{equation}\label{exi:eq16}
A_{3} \leq C \epsilon^{-2-a_{2}} \bigg[ \omega_{2}(|x_{\delta}-y_{\delta}|) + \omega_{3}(|x_{\delta}-y_{\delta}|)\bigg].
\end{equation}
Combining \eqref{exi:eq13}--\eqref{exi:eq16}, we obtain
\begin{equation}
\label{eq3.26}
\begin{aligned}
D_{2} \leq C\epsilon^{-a_{2}}|\omega_{1}(|x_{\delta}-y_{\delta}|)|\left(|\ln^{\epsilon}| +\frac{|x_{\delta}-y_{\delta}|}{\delta}   \right) \\
+C \epsilon^{-2-a_{2}} \bigg[ \omega_{2}(|x_{\delta}-y_{\delta}|) + \omega_{3}(|x_{\delta}-y_{\delta}|)\bigg].
\end{aligned}
\end{equation}

Inserting \eqref{exi:eq12} and
\eqref{eq3.26} into
\eqref{exi:eq11}, we arrive
\begin{equation}\label{exi:eq17}
\begin{aligned}
0<\epsilon a_{0}  \leq \epsilon ^{-\theta_{2}} \omega(|x_{\delta}-y_{\delta}|)&+C\epsilon^{-a_{2}}|\omega_{1}(|x_{\delta}-y_{\delta}|)|\left(|\ln^{\epsilon}| +\frac{|x_{\delta}-y_{\delta}|}{\delta}   \right)  \\
& +C \epsilon^{-2-a_{2}} \bigg[ \omega_{2}(|x_{\delta}-y_{\delta}|) + \omega_{3}(|x_{\delta}-y_{\delta}|)\bigg].
\end{aligned}
\end{equation}
Letting $ \delta \rightarrow 0 $ in \eqref{exi:eq17} and using \eqref{exi:eq3}, we get a contradiction.
\end{proof}

\subsection{Proof of Theorem \ref{pro1}}
In this subsection, we prove the existence of solution to problem \eqref{eq6}.

Following the standard argument in \cite[Lemma 2]{HPRS21} or \cite[Lemma 3.2]{De22}, we can construct a viscosity super-solution $ \overline{\omega}$ and a viscosity sub-solution $ \underline{\omega} $ of the approximating problem \eqref{exi:eq1}. Then using the Perron's method in \cite[Theorem 4.1]{CIL92} and the comparison principle in Lemma \ref{lemma3.2}, we can derive the existence and uniqueness of viscosity solution $ u_{\epsilon}^{v}  $ to the approximating problem \eqref{exi:eq1}, satisfying
$$\underline \omega \le u^v_{\epsilon}\le \overline{\omega} \ \ {\rm in} \  B_1, \ {\rm and} \ \ u^v_\epsilon=g \ {\rm on} \   \partial B_{1}.$$

Before presenting the proof of Theorem \ref{pro1}, we introduce two useful lemmas for the solution $u^v_\epsilon$ of the approximating problem \eqref{exi:eq1}.
\begin{Lemma}
\label{lemma3.3}
The solution $ u_{\epsilon}^{v} $ of the approximating problem \eqref{exi:eq1} is $ C^{0,\gamma_{0}}_{loc}(B_{1}) $ for some $ \gamma_{0} \in (0,1)   $ with the estimate
$$ ||u_{\epsilon}^{v}||_{C^{0,\gamma_{0}}(B')} \leq C=C(n,p,||f||_{L^{\infty}(B_{1})},||g||_{L^{\infty}(B_{1})},\gamma_{0}), \ \ \forall \ B' \subset\subset B_{1}.   $$
\end{Lemma}
The proof of this lemma is analogous to the proof of Lemma \ref{lem4.1}. We shall present the detailed proof of Lemma \ref{lem4.1} in Section \ref{Section 4}. For the sake of brevity, we omit the proof of Lemma \ref{lemma3.3} here.

\begin{Lemma}
\label{lemma3.4}
Define the set $ G = \{ v \in C(\overline{B_{1}})| \underline{\omega} \leq v \leq \overline{\omega}\}$ and the operator $ T : G \rightarrow C(\overline{B_{1}}) $. Given $ v \in G $, and let $ Tv=u_{\epsilon}^{v} $, then the following properties hold:
\begin{itemize}
    \item [{\it (1).}] $ G $ is a closed convex set in $ C(\overline{B_{1}}) $;

    \item [{\it (2).}] $ T(G) $ is a precompact subset in $ C(\overline{B_{1}}) $ and $ T : G \rightarrow G $ is continuous.
\end{itemize}
\end{Lemma}
For the proof of Lemma \ref{lemma3.4}, see \cite[Lemma 3]{HPRS21} and \cite[Proposition 2]{PS22}.

\vspace{2mm}

With Lemma \ref{lemma3.3} and Lemma \ref{lemma3.4} in hand, we now proceed to prove Theorem \ref{pro1}.
\begin{proof}[Proof of Theorem~\ref{pro1}]
Lemma \ref{lemma3.4} allows us to use Schauder's fixed point theorem \cite[Corollary 11.2]{GT01}. More precisely, for any $ \epsilon >0 $, there exists a viscosity solution $ u_{\epsilon} \in  C(\overline{B_{1}}) $ to
\begin{equation*}
F_{\epsilon}^{u_{\epsilon}}(|Du_{\epsilon}|)\left(\epsilon u_{\epsilon} -\Delta u_{\epsilon} -(p-2) \bigg \langle D^{2}u_{\epsilon} \frac{Du_{\epsilon}}{\epsilon+|Du_{\epsilon}|},\frac{Du_{\epsilon}}{\epsilon+|Du_{\epsilon}|}\bigg\rangle \right)=f(x),      \end{equation*}
such that
$$ \underline{\omega} \leq u_{\epsilon} \leq \overline{\omega} \ \ {\rm in} \ B_1, \ {\rm and}\ \ u_{\epsilon} =g \ {\rm on} \  \partial B_{1},$$
where
\begin{equation*}
\begin{aligned}
F_{\epsilon}^{u_{\epsilon}}(|Du_{\epsilon}|) := &(\epsilon+|Du_{\epsilon}|)^{a_{\epsilon}^{u_{\epsilon}}(x)}+ \big[\epsilon+a(x)h_{\epsilon}^{u_{\epsilon}}(x)\big](\epsilon+|Du_{\epsilon}|)^{a_{1}} \\
& + \big[\epsilon+a(x)(1-h_{\epsilon}^{u_{\epsilon}}(x))\big](\epsilon+|Du_{\epsilon}|)^{a_{2}}.
\end{aligned}
\end{equation*}
From Lemma \ref{lemma3.3}, there exists a subsequence $ \{u_{\epsilon_{n}}\} $ such that $ u_{\epsilon_{n}} \rightarrow u $ in $ C(\overline{B_{1}}) $ as $ \epsilon_{n} \rightarrow 0 $. Since $ a_{\epsilon_{n}}^{u_{\epsilon_{n}}}(x) \rightarrow  a_{1}\chi_{\{u>0\}}+a_{2}\chi_{\{u<0\}} $ in $ \left( \{u>0\}\cup \{u<0\}\right)\cap B_{1} $ as $ \epsilon_{n} \rightarrow 0 $, then it follows that $ u $ is a viscosity solution of \eqref{eq6}.
Therefore, the proof of Theorem \ref{pro1} is completed.
\end{proof}

\begin{remark}
The equations \eqref{Section3:eq1} and \eqref{Section3:eq2} can be simplified into \eqref{1.10} and \eqref{1.11}. For this, we make the following claim.

{\bf Claim}: If $ u $ is a viscosity sub-solution and viscosity super-solution to \eqref{Section3:eq1} and \eqref{Section3:eq2}, respectively, then $ u $ is also a viscosity sub-solution and viscosity super-solution to \eqref{1.10} and \eqref{1.11}, respectively.

To prove this claim, let $ u $ be a viscosity sub-solution to \eqref{Section3:eq1}, and let $ \varphi \in C^{2}(B_{1}) $ touch $ u $ from above at $ x_{0} \in B_{1} $, then we have
\begin{align*}
  \min_{i=0,1,2} \bigg  \{ -\big(|D\varphi(x_{0})|^{a_{i}}+a(x_{0})|D\varphi(x_{0})|^{a_{i}}\big)\Delta_{p}^{\rm{N}}\varphi(x_{0}) \bigg \}
   \leq ||f||_{L^{\infty}(B_{1})}.
\end{align*}
Now we want to prove that
\begin{equation*}
  \min_{i=0,1,2} \bigg  \{-|D\varphi(x_{0})|^{a_{i}}\Delta_{p}^{\rm{N}}\varphi(x_{0}) \bigg \}  \leq ||f||_{L^{\infty}(B_{1})}.
\end{equation*}
Note that, if $ \Delta_{p}^{\rm{N}}\varphi(x_{0}) \geq 0 $, then the conclusion is trivial, hence we consider the case when $ \Delta_{p}^{\rm{N}}\varphi(x_{0}) < 0 $. But now, since $ |D\varphi(x_{0})|^{a_{i}} \leq        (1+a(x_{0}))|D\varphi(x_{0})|^{a_{i}}, i=0,1,2 $, the result again follows immediately. The implication \eqref{Section3:eq2} implies \eqref{1.11} follows in an identical way. This observation can simplify the complexity of calculations in section \ref{Section 4}.
\end{remark}

\vspace{3mm}

\section{Local $ C^{1,\alpha'}$ regularity of problem \eqref{eq6} }\label{Section 4}
     ~~~~In this section, our main goal is to prove interior $ C^{1,\alpha'}$ regularity of solution to problem \eqref{eq6}. The proof consists of three steps. In Section \ref{Section 4.1}, we prove a H\"{o}lder regularity result of solution to the perturbed equation using Ishii-Lions approach for large and small slopes. In Section \ref{Section 4.2}, the improvement of flatness lemma is given via scaling techniques and the compactness argument. In Section \ref{Section 4.3}, we use the geometric iteration and conclude the H\"older regularity of the gradient. In Section \ref{Section 4.4}, a direct consequence of Theorem \ref{pro2} is presented.

Now we are in a position to state the main result of this section.
\begin{Theorem}
\label{pro2}
Assume that $ u $ is a viscosity sub-solution to \eqref{1.10} and $ u $ is a viscosity super-solution to \eqref{1.11}, then $ u \in C^{1,\alpha'}_{loc}(B_{1}) $ for some $ \alpha' \in (0,1)$ with the estimate
\begin{equation*}
\label{eq10}
||u||_{C^{1,\alpha'}(B)} \leq C\bigg(||u||_{L^{\infty}(B_{1})} + \max \{||f||_{L^{\infty}(B_{1})}, ||f||_{L^{\infty}(B_{1})}^{\frac{1}{1+a_{1}}}, ||f||_{L^{\infty}(B_{1})}^{\frac{1}{1+a_{2}}}   \}\bigg), \ \ \forall \ B \subset \subset B_{1}.
\end{equation*}
where $ C $ is a positive constant depending only on $ n, p, a_{1}$, $a_{2} $ and $  \text{dist}(B, \partial B_{1})^{-1-\alpha'} $.
\end{Theorem}

\subsection{H\"{o}lder regularity}\label{Section 4.1}
First, we provide local H\"{o}lder continuity of viscosity solution to
\begin{align}\label{loc:eq1}
- |Du+\xi|^{a_{1}\chi_{\{u>0\}}+a_{2}\chi_{\{u<0\}}}
 \Delta_{p,\xi}^{\rm{N}}u=f(x) \ \text{in} \ B_{1} ,
\end{align}
where $ \xi $ is an arbitrary vector in $ \mathbb{R}^{n} $ and
\begin{equation}\label{loc:eq2}
\Delta_{p,\xi}^{\rm{N}}u= \Delta u + (p-2)\bigg\langle D^{2}u \frac{Du+\xi}{|Du+\xi|},\frac{Du+\xi}{|Du+\xi|}\bigg\rangle, \ \ 1<p<\infty .
\end{equation}
In light of the discussions in \cite[Section 2]{De22}, we find that a viscosity solution of equation \eqref{loc:eq1} is both a viscosity sub-solution to
\begin{align}\label{loc:eq3}
\begin{split}
\min \bigg \{-\Delta_{p,\xi}^{\rm{N}}u, -|Du+\xi|^{a_{1}}\Delta_{p,\xi}^{\rm{N}}u, -|Du+\xi|^{a_{2}}\Delta_{p,\xi}^{\rm{N}}u \bigg \} = ||f||_{L^{\infty}(B_{1})}
\end{split}
\end{align}
and a viscosity super-solution to
\begin{align}\label{loc:eq4}
\begin{split}
\max \bigg \{-\Delta_{p,\xi}^{\rm{N}}u, -|Du+\xi|^{a_{1}}\Delta_{p,\xi}^{\rm{N}}u,  -|Du+\xi|^{a_{2}}\Delta_{p,\xi}^{\rm{N}}u \bigg \} = -||f||_{L^{\infty}(B_{1})} .
\end{split}
\end{align}

Now we demonstrate that the viscosity solutions to the perturbed equations \eqref{loc:eq3} and \eqref{loc:eq4} are Lipschitz continuous if $ \xi $ is large enough.
\begin{Lemma}\label{lem4.1}
Let $ u \in C(B_{1}) $ be a viscosity sub-solution to \eqref{loc:eq3} and a viscosity super-solution to \eqref{loc:eq4}. If $ |\xi|>A_{0} $ with $ A_{0} >0 $ depending on $ n, p, a_{1}, a_{2} $, then $ u \in C^{0,1}_{loc}(B_{1})$. Moreover, for every $ r \in (0,1) $, there exists a constant $ C= C(n, p, a_{1}, a_{2}, ||u||_{L^{\infty}(B_{1})},||f||_{L^{\infty}(B_{1})})>0 $ such that for all $ x,y \in B_{r} $,
\begin{equation*}
|u(x)-u(y)|\leq C|x-y|.
\end{equation*}
\end{Lemma}
\begin{proof}
Fixing $ r \in (0,1) $ and $ x_{0} \in B_{1/2}   $, we are going to show that there exist two positive constants $ L_{1} ,L_{2} $ such that
\begin{equation}
\label{eq4.600}
M:=\sup_{x,y \in B_{r}(x_{0})} \left[ u(x)-u(y)-L_{1}h(|x-y|)-L_{2} (|x-x_{0}|^{2}+|y-x_{0}|^{2})\right ]  \leq 0,
\end{equation}
where
\begin{equation*}
\label{eq4.7}
h(t)=\left\{
     \begin{aligned}
     &   t-\frac{1}{2}t^{2} , 0\leq t \leq 1,        \\
     &   h(1) , \ \    t > 1          ,                  \\
     \end{aligned}
     \right.
\end{equation*}
and $ A_{0} >0 $ is to be determined.

For the case $ |\xi| > A_{0} $, by contradiction, we assume that for any $ L_{1}, L_{2} >0  $, the positive maximum $ M  $ is attained at $ (\overline{x},\overline{y}) \in B_{r}(x_{0})    $, that is to say,
\begin{equation}
\label{eq4.8}
L_{1}h(|\overline{x}-\overline{y}|)+L_{2} (|\overline{x}-x_{0}|^{2}+|\overline{y}-x_{0}|^{2}) \leq 2||u||_{L^{\infty}(B_{1})} .
\end{equation}
For simplicity, we denote
$$ \varphi (x,y):= L_{1}h(|x-y|)+L_{2} (|x-x_{0}|^{2}+|y-x_{0}|^{2}) $$
and
$$ \phi(x,y):=  u(x)-u(y)-\varphi (x,y)  .    $$
Notice that $ \overline{x} \neq \overline{y}    $. Indeed, if $ \overline{x} = \overline{y} $, then it is obvious that $ M \leq 0 $, which contradicts with the assumption that the positive maximum $ M $ is attained at $(\bar x, \bar y)$.

Now choosing $ L_{2} =\frac{32}{r^{2}}||u||_{L^{\infty}(B_{1})}  $, this choice of $L_2$ together with (\ref{eq4.8}) yield that
\begin{equation*}
|\overline{x}-x_{0}|+|\overline{y}-y_{0}| <  \frac{1}{2}r,
\end{equation*}
which implies that $ \overline{x},\overline{y}\in B_{r}(x_{0}) $.

By Ishii-Lions Lemma \cite[Theorem 3.2]{CIL92}, we can ensure the existence of limiting subjet $ (q_{\overline{x}},X)$ of $ u $ at $ \overline{x} $ and limiting superjet $ (q_{\overline{y}},Y)$ of $ u $ at $ \overline{y} $ and the following matrix inequality
\begin{equation}\label{loc:eq5}
\begin{pmatrix}
X  &   0   \\
0  &    -Y
\end{pmatrix}
\leq
\begin{pmatrix}
Z   &   -Z   \\
-Z  &    Z
\end{pmatrix}
 + (2L_{2}+\delta)\textbf{I}\rm{d}_{n}, \ \  0 < \delta \ll 1 ,
\end{equation}
where
\begin{equation*}
\label{eq4.10}
Z=L_{1}h''(|\overline{x}-\overline{y}|)\frac{\overline{x}-\overline{y}}{|\overline{x}-\overline{y}|}\otimes \frac{\overline{x}-\overline{y}}{|\overline{x}-\overline{y}|} +\frac{L_{1}h'(|\overline{x}-\overline{y}|)}{|\overline{x}-\overline{y}|}\left( \textbf{I}\rm{d}_{n}- \frac{\overline{x}-\overline{y}}{|\overline{x}-\overline{y}|}\otimes \frac{\overline{x}-\overline{y}}{|\overline{x}-\overline{y}|} \right)
\end{equation*}
and
\begin{equation}\label{loc:eq6}
\left\{
     \begin{aligned}
     & q_{\overline{x}}:=\partial_{\overline{x}}\varphi(\overline{x},\overline{y}) =L_{1}h'(|\overline{x}-\overline{y}|)\frac{\overline{x}-\overline{y}}{|\overline{x}-\overline{y}|}+2L_{2}(\overline{x}-x_{0}) ,      \\
     & q_{\overline{y}}:=-\partial_{\overline{y}}\varphi(\overline{x},\overline{y}) =L_{1}h'(|\overline{x}-\overline{y}|)\frac{\overline{x}-\overline{y}}{|\overline{x}-\overline{y}|}-2L_{2}(\overline{y}-x_{0}) .           \\
     \end{aligned}
     \right.
\end{equation}
Before two viscosity inequalities are given, for simplicity, we denote
\begin{equation}\label{loc:eq7}
\left\{
     \begin{aligned}
     &  F_{q_{\overline{x}}}(X):=-\text{Tr}\left( (\textbf{I}\rm{d}_{n}+(p-2)\frac{q_{\overline{x}}+\xi}{|q_{\overline{x}}+\xi|}\otimes \frac{q_{\overline{x}}+\xi}{|q_{\overline{x}}+\xi|})X\right) := -\text{Tr}\left(A(\eta_{1})X\right) ,        \\
     &   F_{q_{\overline{y}}}(Y):=-\text{Tr}\left( (\textbf{I}\rm{d}_{n}+(p-2)\frac{q_{\overline{y}}+\xi}{|q_{\overline{y}}+\xi|}\otimes \frac{q_{\overline{y}}+\xi}{|q_{\overline{y}}+\xi|})Y\right) := -\text{Tr}\left(A(\eta_{2})Y\right) ,          \\
     & Q_{i}(q_{\overline{x}},\xi) :=|q_{\overline{x}}+\xi|^{a_{i}}, \ \ i=1,2 ,  \\
     & Q_{i}(q_{\overline{y}},\xi) :=|q_{\overline{y}}+\xi|^{a_{i}}, \ \ i=1,2  ,    \\
     \end{aligned}
     \right.
\end{equation}
where
\begin{equation}\label{loc:eq8}
\left\{
     \begin{aligned}
     & A(\eta_{1})= \textbf{I}\rm{d}_{n}+(p-2)\frac{\eta_{1}}{|\eta_{1}|}\otimes \frac{\eta_{1}}{|\eta_{1}|}  , \ \ \eta_{1}= q_{\overline{x}}+\xi ,     \\
     & A(\eta_{2})= \textbf{I}\rm{d}_{n}+(p-2)\frac{\eta_{2}}{|\eta_{2}|}\otimes \frac{\eta_{2}}{|\eta_{2}|}  , \ \ \eta_{2}= q_{\overline{y}}+\xi .          \\
     \end{aligned}
     \right.
\end{equation}
Then from \eqref{loc:eq7} and \eqref{loc:eq8}, we obtain the following two viscosity inequalities:
\begin{equation*}
\label{eq4.14}
\min \bigg \{ F_{q_{\overline{x}}}(X), Q_{1}(q_{\overline{x}},\xi) F_{q_{\overline{x}}}(X), Q_{2}(q_{\overline{x}},\xi) F_{q_{\overline{x}}}(X)     \bigg \} \leq ||f||_{L^{\infty}(B_{1})} ,
\end{equation*}
\begin{equation*}
\label{eq4.15}
\max \bigg \{ F_{q_{\overline{y}}}(Y), Q_{1}(q_{\overline{y}},\xi) F_{q_{\overline{y}}}(Y), Q_{2}(q_{\overline{y}},\xi) F_{q_{\overline{y}}}(Y)     \bigg \} \geq -||f||_{L^{\infty}(B_{1})}.
\end{equation*}

Next, we consider all the possible cases.

Suppose
\begin{equation*}
\min \bigg \{ F_{q_{\overline{x}}}(X), Q_{1}(q_{\overline{x}},\xi) F_{q_{\overline{x}}}(X), Q_{2}(q_{\overline{x}},\xi) F_{q_{\overline{x}}}(X)     \bigg \} =F_{q_{\overline{x}}}(X)\leq ||f||_{L^{\infty}(B_{1})} ,
\end{equation*}
\begin{equation*}
\max \bigg \{ F_{q_{\overline{y}}}(Y), Q_{1}(q_{\overline{y}},\xi) F_{q_{\overline{y}}}(Y), Q_{2}(q_{\overline{y}},\xi) F_{q_{\overline{y}}}(Y)     \bigg \} =F_{q_{\overline{y}}}(Y)\geq -||f||_{L^{\infty}(B_{1})}  .
\end{equation*}
In this case, we have
\begin{equation}\label{loc:eq9}
F_{q_{\overline{y}}}(Y)-F_{q_{\overline{x}}}(X)  \geq -2||f||_{L^{\infty}(B_{1})}.
\end{equation}
Similarly, we can get the rest of the situations separately:
\begin{equation}\label{loc:eq10}
\left\{
     \begin{aligned}
     & F_{q_{\overline{y}}}(Y)-F_{q_{\overline{x}}}(X)  \geq - \bigg(1+Q_{i}^{-1}(q_{\overline{x}},\xi)\bigg)||f||_{L^{\infty}(B_{1})} , \ \  i=1,2   ,     \\
     &  F_{q_{\overline{y}}}(Y)-F_{q_{\overline{x}}}(X)  \geq - \bigg(1+Q_{i}^{-1}(q_{\overline{y}},\xi)\bigg)||f||_{L^{\infty}(B_{1})} ,\ \  i=1,2   ,     \\
     &   F_{q_{\overline{y}}}(Y)-F_{q_{\overline{x}}}(X)  \geq - \bigg(Q_{j}^{-1}(q_{\overline{x}},\xi)+Q_{i}^{-1}(q_{\overline{y}},\xi)\bigg)||f||_{L^{\infty}(B_{1})} ,\ \  i,j=1,2   .      \\
     \end{aligned}
     \right.
\end{equation}
From \eqref{loc:eq6}, we see that
\begin{equation*}
\label{eqp}
|q_{\overline{x}}|,|q_{\overline{y}}|\leq L_{1}  +2L_{2}.
\end{equation*}
By choosing $ A_{0} =3L_{1} +2L_{2}   $, it follows that
\begin{equation}\label{loc:eq11}
|q_{\overline{y}}+\xi|, |q_{\overline{x}}+\xi| \geq |\xi|-|q_{\overline{x}}|\geq 2L_{1} .
\end{equation}
Combining \eqref{loc:eq9}, \eqref{loc:eq10} and \eqref{loc:eq11}, we obtain
\begin{equation}\label{loc:eq12}
 F_{q_{\overline{y}}}(Y)-F_{q_{\overline{x}}}(X) \geq -C(2L_{1})^{-a_{2}}||f||_{L^{\infty}(B_{1})}.
\end{equation}
This gives a lower bound of $  F_{q_{\overline{y}}}(Y)-F_{q_{\overline{x}}}(X)   $.

Next, we shall derive an upper bound of $  F_{q_{\overline{y}}}(Y)-F_{q_{\overline{x}}}(X)$ using \eqref{loc:eq7} and \eqref{loc:eq8}. Indeed, we have
\begin{align}\label{loc:eq13}
\begin{split}
&  F_{q_{\overline{y}}}(Y)-F_{q_{\overline{x}}}(X)=\text{Tr}(A(\eta_{1})X)-\text{Tr}(A(\eta_{2})Y)      \\
 & \ \ \ \ \ \ \ \ \  \ \ \ \ \ \ \ \ \ \ \ \ \ = \underbrace{\text{Tr}\bigg(A(\eta_{1})(X-Y)\bigg)}_{:=\widetilde{A}}+\underbrace{\text{Tr}\bigg((A(\eta_{1})-A(\eta_{2}))Y\bigg)}_{:=B} .
\end{split}
\end{align}
Applying the matrix inequality \eqref{loc:eq5} to the vector $ (\overline{\xi},-\overline{\xi}) \in \mathbb {R}^{2n}   $ where $ \overline{\xi} =\frac{\overline{x}-\overline{y}}{|\overline{x}-\overline{y}|}$, we have
\begin{equation*} \label{eq4.21}
 \big\langle (X-Y)\overline{\xi},-\overline{\xi} \big\rangle \leq 4\overline{\xi}^{T} Z \overline{\xi} +4L_{2}+2\delta= 4L_{2}+2\delta-4L_{1} \leq0  ,
\end{equation*}
provided $ L_{1} $ is large enough, depending on $ L_{2} $ and $ \delta $, which implies that
\begin{equation}\label{loc:eq14}
{\color{blue}{\widetilde{A}}}= \sum_{i=1}^{n} \lambda_{i}(A(\eta_{1}))\lambda_{i}(X-Y) \leq \min \big\{1,p-1\big\}\big(4L_{2}+2\delta-4L_{1}\big).
\end{equation}
Note that
\begin{equation*}
\label{eq4.23}
A(\eta_{1})-A(\eta_{2})=(p-2) \big\{\widetilde{\eta_{1}}\otimes (\widetilde{\eta_{1}}-\widetilde{\eta_{2}})-\widetilde{\eta_{2}}\otimes(\widetilde{\eta_{2}}-\widetilde{\eta_{1}})     \big\} ,
\end{equation*}
where $ \widetilde{\eta_{i}} = \frac{\eta_{i}}{|\eta_{i}|}, \ i=1,2$.
Then we have
\begin{align}\label{loc:eq15}
\begin{split}
B & \leq n|p-2||\widetilde{\eta_{1}}-\widetilde{\eta_{2}}|\big(|\widetilde{\eta_{1}}|+|\widetilde{\eta_{2}}|\big)||Y||    \\
&   \leq 2n|p-2|||Y|||\widetilde{\eta_{1}}-\widetilde{\eta_{2}}|,
\end{split}
\end{align}
where the definitions of $ \widetilde{\eta_{i}} \ (i=1,2)$ are used.

Now we turn to the estimate of $ ||Y|| $ and $ |\widetilde{\eta_{1}}-\widetilde{\eta_{2}}| $ in \eqref{loc:eq15}. Applying the second matrix inequality in \eqref{loc:eq5}  to the vector $ ( \xi, \xi) \in \mathbb {R}^{2n} $ and $ |\xi| =1  $, then
\begin{equation*}
 \langle (X-Y)\xi, \xi \rangle \leq (4L_{2}+2\delta)|\xi|^{2}= 4L_{2}+2\delta ,
\end{equation*}
which implies that
\begin{equation}\label{loc:eq16}
||X-Y|| \leq 4L_{2} +2\delta.
\end{equation}
In a similar way, we have
\begin{equation}\label{locc:eq17}
 ||X|| = \sup_{|\xi|\leq 1}|\xi^{T} X\xi| \leq ||Z|| + 2L_{2} +\delta \leq L_{1} + 2L_{2} + \delta,
\end{equation}
where we have used the fact $ ||Z|| \leq L_{1} $. Combining the previous estimate \eqref{loc:eq16} with \eqref{locc:eq17}, it follows
\begin{equation}\label{locc:eq18}
  ||Y|| \leq L_{1} + 6L_{2} +3\delta.
\end{equation}

By using \eqref{loc:eq8}, \eqref{loc:eq11} and $ |\eta_{1} - \eta_{2}| \leq 4L_{2} $(since $ |\overline{x}+\overline{y}-2x_{0}| \leq 2$), we obtain
\begin{align*}
\label{eq4.26}
\begin{split}
|\widetilde{\eta_{1}}-\widetilde{\eta_{2}}|=\bigg|\frac{\eta_{1}}{|\eta_{1}|}-\frac{\eta_{2}}{|\eta_{2}|}\bigg| \leq \max \bigg\{\frac{\eta_{1}-\eta_{2}}{|\eta_{1}|}, \frac{\eta_{1}-\eta_{2}}{|\eta_{2}|}   \bigg\}
 \leq \frac{2L_{2}}{L_{1}},
\end{split}
\end{align*}
which, together with \eqref{loc:eq15} and \eqref{locc:eq18}, yields that
\begin{equation}\label{loc:eq17}
B \leq 4n|p-2|\left( L_{2} + \frac{L_{2}}{L_{1}} (6L_{2}+3\delta)\right).
\end{equation}
Combining \eqref{loc:eq12}, \eqref{loc:eq13}, \eqref{loc:eq14} and \eqref{loc:eq17}, we get
\begin{align*}
4L_{1}\min\big\{1,p-1\big\} \leq \min\big\{1,p-1\big\} \big(4L_{2}+2\delta\big)  + 4n & |p-2|\left( L_{2} + \frac{L_{2}}{L_{1}} (6L_{2}+3\delta)\right) \\
&+ C(2L_{1})^{-a_{2}}||f||_{L^{\infty}(B_{1})}.
\end{align*}
By choosing $ L_{1} \gg 1 $ large enough, depending on $ n, p, \delta, a_{2}, L_{2}, ||f||_{L^{\infty}(B_{1})} $, we reach a contradiction.

Thereby, we verify \eqref{eq4.600}, which implies that $ u $ is Lipschitz continuous and satisfies
\begin{equation*}
|u(x)-u(y)|\leq C|x-y|,
\end{equation*}
where $ C= C(n, p, a_{1}, a_{2}, ||u||_{L^{\infty}(B_{1})},||f||_{L^{\infty}(B_{1})}) $ is a positive constant.

The proof is now complete.
\end{proof}

The next lemma deals with the other alternative of the norm of $ \xi $, i.e. the case where $ |\xi| \leq A_{0} $.

\begin{Lemma}\label{Sec4:lem4.2}
Let $ u \in C(B_{1}) $ be a viscosity sub-solution to \eqref{loc:eq3} and a viscosity super-solution to \eqref{loc:eq4}. If $ |\xi| \leq A_{0} $ being the same as that in Lemma \ref{lem4.1}, then $ u \in C^{0,\beta}_{loc}(B_{1})$ for some $ 0<\beta<1 $. Moreover, for every $ r \in (0,1) $, there exists a constant $ C= C(n, p, a_{1}, a_{2},||u||_{L^{\infty}(B_{1})},||f||_{L^{\infty}(B_{1})})>0 $ such that for all $ x,y \in B_{r} $,
\begin{equation*}
|u(x)-u(y)|\leq C|x-y|^{\beta}.
\end{equation*}
\end{Lemma}

\begin{proof}
If $ |Du| \geq 2 A_{0} $, then $ |Du+ \xi|\geq \big||Du|-|\xi|\big| \geq A_{0} $. Noticing that $ u $ is a viscosity sub-solution to
\begin{equation*}
 \mathcal{M}^{+}_{\lambda,\Lambda}(D^{2}u) + C_{0}^{-1} ||f||_{L^{\infty}(B_{1})}  = 0,
\end{equation*}
and $ u $ is a viscosity super-solution to
\begin{equation*}
   \mathcal{M}^{-}_{\lambda,\Lambda}(D^{2}u) - C_{0}^{-1}||f||_{L^{\infty}(B_{1})} = 0,
\end{equation*}
where $ C_{0} $ is a positive constant depending only on $ A_{0}, a_{1}, a_{2} $. Indeed, take $ \varphi \in C^{2}(B_{1}) $ such that $ u - \varphi $ has a local maximum at $ x_{0} \in \{|Du| \geq 2A_{0}\} $, then $ |D\varphi(x_{0})+\xi|^{a_{i}} \geq A_{0}^{a_{i}}, i=1,2 $. We assume that $  -\Delta_{p,\xi}^{\rm{N}}\varphi(x_{0}) \geq 0 $, since the case $  -\Delta_{p,\xi}^{\rm{N}}\varphi(x_{0}) \leq 0 $ is trivial. From \eqref{loc:eq3} we have
\begin{equation*}
  \Delta_{p,\xi}^{\rm{N}}\varphi(x_{0}) \geq -C_{0}^{-1} ||f||_{L^{\infty}(B_{1})},
\end{equation*}
where $ C_{0} = \min \{1, A_{0}^{a_{1}}, A_{0}^{a_{2}}\} $. As a consequence, in the set $ |Du| \geq 2A_{0} $, $ u $ is a viscosity sub-solution to
 $ \mathcal{M}^{+}_{\lambda,\Lambda}(D^{2}u) + C_{0}^{-1} ||f||_{L^{\infty}(B_{1})}  = 0 $.
In a similar way, we can prove that $ u $ is a viscosity super-solution to $ \mathcal{M}^{-}_{\lambda,\Lambda}(D^{2}u) - C_{0}^{-1}||f||_{L^{\infty}(B_{1})} = 0 $.

Using the result of \cite[Theorem 1.1]{LS16} (see also \cite[Proposition 2]{HPRS21}), there exists $ \beta \in (0,1) $ such that $ u \in C^{0,\beta}_{loc}(B_{1}) $ and
\begin{equation*}
|u(x)-u(y)|\leq C|x-y|^{\beta},
\end{equation*}
where $ C =  C(n, p, a_{1}, a_{2},||u||_{L^{\infty}(B_{1})},||f||_{L^{\infty}(B_{1})}) $ is a positive constant.
\end{proof}

By combining Lemma \ref{lem4.1} and Lemma \ref{Sec4:lem4.2}, we obtain a compactness result for viscosity solutions to the perturbed equations \eqref{loc:eq3} and \eqref{loc:eq4}.

\begin{Proposition}\label{Sec4:prop1}
Let $ u \in C(B_{1}) $ be a viscosity sub-solution to \eqref{loc:eq3} and a viscosity super-solution to \eqref{loc:eq4}. Then $ u \in C^{0,\beta}_{loc}(B_{1})$ for some $ 0<\beta<1 $. Moreover, for every $ r \in (0,1) $, there exists a constant $ C= C(n, p, a_{1}, a_{2},||u||_{L^{\infty}(B_{1})},||f||_{L^{\infty}(B_{1})})>0 $ such that for all $ x,y \in B_{r} $,
\begin{equation*}
|u(x)-u(y)|\leq C |x-y|^{\beta}.
\end{equation*}
\end{Proposition}

\subsection{The improvement of flatness lemma}\label{Section 4.2}
For convenience, we perform a scaling argument for \eqref{loc:eq3} and \eqref{loc:eq4}.

Letting
$$ v(x)=\frac{u(x_{0}+\zeta x)}{K},   $$
where $ \zeta $ and $ K $ are positive constants to be determined, direct computations yield that $ v $ is a viscosity sub-solution to
\begin{align}\label{loc:eq18}
\begin{split}
&\min \big \{-\Delta_{p,\widetilde{\xi}}^{\rm{N}}v, -Q_{1}(Dv,\widetilde{\xi};x_{0})\Delta_{p,\widetilde{\xi}}^{\rm{N}}v, \\
& \ \ \ \ \ \ \ \ \ \ \ \ \ \ \ \ \ \ -Q_{2}(Dv,\widetilde{\xi};x_{0})\Delta_{p,\widetilde{\xi}}^{\rm{N}}v \big \} = ||\widehat{f}||_{L^{\infty}(B_{1})}
\end{split}
\end{align}
and $ v $ is a viscosity super-solution to
\begin{align}\label{loc:eq19}
\begin{split}
&\max \big \{-\Delta_{p,\widetilde{\xi}}^{\rm{N}}v, -Q_{1}(Dv,\widetilde{\xi};x_{0})\Delta_{p,\widetilde{\xi}}^{\rm{N}}v, \\
& \ \ \ \ \ \ \ \ \ \ \ \ \ \ \ \ \ \ -Q_{2}(Dv,\widetilde{\xi};x_{0})\Delta_{p,\widetilde{\xi}}^{\rm{N}}v \big \} =-||\widehat{f}||_{L^{\infty}(B_{1})},
\end{split}
\end{align}
where
\begin{equation*}
\left\{
     \begin{aligned}
     & Q_{i}(z,\widetilde{\xi};x_{0}) := |z+\widetilde{\xi}|^{a_{i}}, \ i=1,2, \ \widetilde{\xi}=\frac{\zeta}{K}\xi    ,     \\
     &  \Delta_{p,\widetilde{\xi}}^{\rm{N}}v:= \Delta v +(p-2) \bigg\langle D^{2}v\frac{Dv+\widetilde{\xi}}{|Dv+\widetilde{\xi}|}, \frac{Dv+\widetilde{\xi}}{|Dv+\widetilde{\xi}|}\bigg\rangle   ,     \\
     & ||\widehat{f}||_{L^{\infty}(B_{1})}:=\max \bigg \{\frac{\zeta^{2}}{K},\frac{\zeta^{2+a_{1}}}{K^{1+a_{1}}},\frac{\zeta^{2+a_{2}}}{K^{1+a_{2}}}\bigg\}        ||f||_{L^{\infty}(B_{1})}  .    \\
     \end{aligned}
     \right.
\end{equation*}

By choosing $ \zeta = \delta^{\frac{1}{2}} (0<\delta <1)$ and
$$ K=2\bigg(||u||_{L^{\infty}(B_{1})}+\max \{||f||_{L^{\infty}(B_{1})}, ||f||_{L^{\infty}(B_{1})}^{\frac{1}{1+a_{1}}}, ||f||_{L^{\infty}(B_{1})}^{\frac{1}{1+a_{2}}}   \} \bigg),   $$
it is not difficult to verify that
\begin{equation*}
\label{eq4.22}
 ||v||_{L^{\infty}(B_{1})} \leq \frac{1}{2}, \ \
{\rm and} \ \ ||\widehat{f}||_{L^{\infty}(B_{1})}   \leq \delta   .
\end{equation*}
Therefore, in what follows, we will always assume that the conditions
\begin{equation}\label{loc:eq20}
 ||u||_{L^{\infty}(B_{1})} \leq \frac{1}{2}, \ \  {\rm and} \ \ ||f||_{L^{\infty}(B_{1})}   \leq \delta
\end{equation}
hold.

The H\"{o}lder estimate in Proposition \ref{Sec4:prop1} on viscosity solution provides the compactness result with respect to uniform convergence, which plays an important role in the the improvement of flatness lemma.

Before stating the improvement of flatness lemma,  we first recall the H\"{o}lder estimate for gradient to homogeneous normalized $ p$-Laplacian equation. We refer the reader to \cite[Lemma 3.2]{AME17} and \cite[Lemma 3.4]{AE18}.

\begin{Lemma}
\label{lem4.21}
Let $ v $ be a viscosity solution of
$$ -\Delta v -(p-2)\bigg\langle D^{2}v\frac{Dv+\xi}{|Dv+\xi|}, \frac{Dv+\xi}{|Dv+\xi|}\bigg\rangle =0 \ \text{in} \ B_{1}, \ \ \xi \in  \mathbb{R}^{n} , $$
with $ ||v||_{L^{\infty}(B_{1})} \leq \frac{1}{2} $. For all $ 0\leq r  \leq \frac{1}{2} $, there exists constants $ C_{0}=C_{0}(p,n) >0 $ and $ \beta_{1} =\beta_{1}(p,n) >0 $ such that
$$ ||v||_{C^{1,\beta_{1}}(B_{r})} \leq C_{0} . $$
\end{Lemma}

Now, we state the improvement of flatness lemma as follows.

\begin{Lemma}
\label{lem4.22}
Suppose $ u $ is a viscosity sub-solution to \eqref{loc:eq3} and $ u $ is viscosity super-solution to \eqref{loc:eq4}. Then there exists $ 0<\rho<1$ and $ \delta >0 $, depending only on $ p,n$ and $a_{2}$, such that if $ ||u||_{L^{\infty}(B_{1})} \leq \frac{1}{2}$ and $  ||f||_{L^{\infty}(B_{1})} \leq \delta $, the inequality
$$  \mathop{osc}\limits_{B_{\rho}}\big(u-\vec{q}\cdot x\big) \leq \frac{1}{2}\rho$$
holds for some $ \vec{q} \in \mathbb{R}^{n} $.
\end{Lemma}
\begin{proof}
We prove by contradiction. Suppose that there exists a sequence of function $ \{u_{j}\}_{j=1}^{\infty} $ with $ ||u_{j}||_{L^{\infty}(B_{1})} \leq \frac{1}{2}   $, and sequences of function $ \{f_{j}\}_{j=1}^{\infty}, \{a_{j}\}_{j=1}^{\infty}$ with $ ||f_{j}||_{L^{\infty}(B_{1})}\leq \frac{1}{j}  $, and a sequence of vector $ \{\xi_{j}\}_{j=1}^{\infty}   $ such that $ \{u_{j}\}_{j=1}^{\infty} $ is viscosity sub-solution to
\begin{align}\label{loc:eq21}
\min \big \{-\Delta_{p,\xi_{j}}^{\rm{N}}u_{j}, -|Du_{j}+\xi_{j}|^{a_{1}}\Delta_{p,\xi_{j}}^{\rm{N}}u_{j}, -|Du_{j}+\xi_{j}|^{a_{2}}\Delta_{p,\xi_{j}}^{\rm{N}}u_{j} \big \} = \frac{1}{j}
\end{align}
and $ \{u_{j}\}_{j=1}^{\infty} $ is viscosity super-solution to
\begin{align}\label{loc:eq22}
\max \big \{-\Delta_{p,\xi_{j}}^{\rm{N}}u_{j}, -|Du_{j}+\xi_{j}|^{a_{1}}\Delta_{p,\xi_{j}}^{\rm{N}}u_{j}, -|Du_{j}+\xi_{j}|^{a_{2}}\Delta_{p,\xi_{j}}^{\rm{N}}u_{j} \big \} = -\frac{1}{j},
\end{align}
but for any $ \vec{q} \in \mathbb{R}^{n}  $ and $ \rho \in (0,1)$, the following inequality holds
\begin{equation}\label{loc:eq23}
\mathop{osc}\limits_{B_{\rho}}\big(u_{j}-\vec{q}\cdot x\big) >\frac{1}{2}\rho.
\end{equation}
By Proposition \ref{Sec4:prop1} and the compactness argument, we have that
$$ u_{j} \rightarrow u_{\infty} \ \ \text{locally uniformly in $ B_{1} $}. $$
Now it remains to show that $ u_{\infty} \in C^{1,\widehat{\beta}}_{loc}(B_{1})   $ for some $ \widehat{\beta} \in(0,1)$. Next, we treat separately the cases when $ \{\xi_{j}\}_{j=1}^{\infty}   $ is unbounded or bounded.

Case 1: $ \{\xi_{j}\}_{j=1}^{\infty}   $ is an unbounded sequence. We can extract a subsequence from $ e_{j}:= \frac{\xi_{j}}{|\xi_{j}|}  $ such that $ e_{j}  \rightarrow e_{\infty} $ with $ |e_{\infty}|=1   $. Next our goal is to prove $ u_{\infty} $ is a viscosity solution to
\begin{equation}\label{loc:eq24}
-\Delta u_{\infty} -(p-2)\big \langle D^{2}u_{\infty}e_{\infty},e_{\infty} \big \rangle =0 \ \ \text{in} \ \ B_{1}.
\end{equation}
We prove that $ u_{\infty} $ is a viscosity super-solution of \eqref{loc:eq24}.
Let $ \varphi$ be a function in $C^{2}(B_{1})$ such that $ u_{\infty}-\varphi   $ attains its local minimum at $ x_{0} \in B_{1}    $.
We assume by contradiction that
\begin{equation}\label{loc:eq25}
-\Delta \varphi(x_{0}) -(p-2)\big \langle D^{2}\varphi(x_{0})e_{\infty},e_{\infty} \big \rangle <0.
\end{equation}
There exists a sequence $ \{x_{j}\}_{j\in \mathbb{N}}$ such that $ x_{j} \rightarrow x_{0}   $ and $ u_{j}-\varphi $ reaches a local minimum at $ x_{j} \in B_{1}   $. Then by the equation \eqref{loc:eq22} satisfied by $ u_{j} $ in the viscosity sense, we have
\begin{align}\label{loc:eq26}
\begin{split}
\max \big \{ - & \Delta_{p,\xi_{j}}^{\rm{N}}\varphi(x_{j}), -|D\varphi(x_{j})+\xi_{j}|^{a_{1}}\Delta_{p,\xi_{j}}^{\rm{N}}\varphi(x_{j}), -|D\varphi(x_{j})+\xi_{j}|^{a_{2}}\Delta_{p,\xi_{j}}^{\rm{N}}\varphi(x_{j}) \big \}   \\
& \geq -\frac{1}{j}.
\end{split}
\end{align}
Noticing that $ x_{j} \rightarrow x_{0} $ and $ e_{j}  \rightarrow e_{\infty} $, then from \eqref{loc:eq25}, for sufficiently large $ j $, we deduce
\begin{equation}\label{loc:eq27}
-\Delta_{p,\xi_{j}}^{\rm{N}}\varphi(x_{j}) <0,
\end{equation}
which implies that \eqref{loc:eq26} can be written as
\begin{align}\label{loc:eq28}
-\Delta_{p,\xi_{j}}^{\rm{N}}\varphi(x_{j}) \geq -\frac{j^{-1}}{E^{*}},
\end{align}
where
$$ E^{*} := \min \big \{1, |D\varphi(x_{j})+\xi_{j}|^{a_{1}},|D\varphi(x_{j})+\xi_{j}|^{a_{2}} \big \}.
$$
Taking $ j $ large enough, we have $ |D\varphi(x_{j})+\xi_{j}| >1 $. Then from the expression of $ E^{*} $ and \eqref{loc:eq28}, we obtain
$$ -\Delta \varphi(x_{0}) -(p-2)\big \langle D^{2}\varphi(x_{0})e_{\infty},e_{\infty} \big \rangle \geq 0 , $$
which contradicts with \eqref{loc:eq25}. Hence, we have proved that $ u_{\infty} $ is a viscosity super-solution of \eqref{loc:eq24}.
The proof of $u_\infty$ being a viscosity sub-solution of \eqref{loc:eq24} is similar. Therefore, $ u_{\infty} $ is a viscosity solution to \eqref{loc:eq24}.

From Remark \ref{rk6}, we see that \eqref{loc:eq24} is a linear uniformly elliptic equation with constant coefficients. Then invoking the regularity result in \cite[Corollary 5.7]{CC95}, we derive $ u_{\infty} \in C^{1,\widetilde{\alpha}}_{loc}(B_{1})   $ for some $ 0< \widetilde{\alpha} <1 $.

Case 2: $ \{\xi_{j}\}_{j=1}^{\infty}$ is a bounded sequence, which allows us to select a subsequence $\{\xi_{j_k}\}_{k=1}^\infty$ from $ \{\xi_{j}\}_{j=1}^{\infty} $ such that $ \xi_{j_k} \rightarrow \xi_{\infty}$ as $k\rightarrow +\infty$. For convenience, we still denote the subsequence $\{\xi_{j_k}\}_{k\in \mathbb{N}}$ by $\{\xi_j\}_{j\in \mathbb{N}}$. In the following, we shall prove $ u_{\infty} $ is a viscosity solution to
\begin{equation}\label{loc:eq29}
-\Delta u_{\infty} -(p-2)\bigg \langle D^{2}u_{\infty}\frac{Du_{\infty}+\xi_{\infty}}{|Du_{\infty}+\xi_{\infty}|},\frac{Du_{\infty}+\xi_{\infty}}{|Du_{\infty}+\xi_{\infty}|} \bigg \rangle =0\ \text{in} \ B_{1}.
\end{equation}
To this aim,  let $ \varphi \in C^{2}(B_{1}) $ be a function such that $ u_{\infty}-\varphi   $ attains its local minimum at $ x_{0} \in B_{1}    $. For simplicity, we suppose that
$$ \varphi(x) = \frac{1}{2}(x-x_{0})^{T}A(x-x_{0}) +\vec{b} \cdot (x-x_{0}) +u_{\infty}(x_{0}).$$
Now we distinguish two cases according to $ |\vec{b}+\xi_{\infty}| >0    $ and $ |\vec{b}+\xi_{\infty}| =0 $. For the case when $ |\vec{b}+\xi_{\infty}| >0    $, we assume by contradiction that
\begin{equation}\label{loc:eq30}
-\Delta \varphi(x_{0}) -(p-2)\bigg \langle D^{2}\varphi(x_{0})\frac{\vec{b}+\xi_{\infty}}{|\vec{b}+\xi_{\infty}|},\frac{\vec{b}+\xi_{\infty}}{|\vec{b}+\xi_{\infty}|} \bigg \rangle <0.
\end{equation}
We can carry out the same procedure as Case 1 above. In fact, taking $ j $ large enough, we arrive at $ |D\varphi(x_{j})+\xi_{j}| \geq\frac{1}{2}|\vec{b}+\xi_{\infty}| >0 $, then combining \eqref{loc:eq26}, \eqref{loc:eq27} and \eqref{loc:eq28}, we derive
 $$ -\Delta \varphi(x_{0}) -(p-2)\bigg \langle D^{2}\varphi(x_{0})\frac{\vec{b}+\xi_{\infty}}{|\vec{b}+\xi_{\infty}|},\frac{\vec{b}+\xi_{\infty}}{|\vec{b}+\xi_{\infty}|} \bigg \rangle \geq 0, $$
 which contradicts to \eqref{loc:eq30}. Hence $ u_{\infty}$ is a viscosity super-solution to \eqref{loc:eq29}. Similarly, we can prove that $ u_{\infty} $ is a viscosity sub-solution to \eqref{loc:eq29}. Therefore, \eqref{loc:eq29} is proved.

For the case when $ |\vec{b}+\xi_{\infty}| =0$, there are two possibilities, namely,

Case $ 2a $: $ \vec{b}+\xi_{\infty} =0, |\vec{b}|, |\xi_{\infty}|>0 $.

Case $ 2b $: $ |\vec{b}|=|\xi_{\infty}|=0 $.

We first consider Case $ 2a $. If there exists a subsequence $ \{x_{j'}\} $ such that $ |D\varphi(x_{j'})+\xi_{j'}| >0 $, then we can repeat process above to show \eqref{loc:eq29} holds. If such a subsequence does not exist, we want to show that
\begin{equation}\label{loc:eq31}
-\Delta\varphi(x_{0}) -(p-2)\lambda_{\min}(D^{2}\varphi(x_{0}))\geq0, \ \ p\geq 2,
\end{equation}
and
\begin{equation}\label{loc:eq32}
-\Delta \varphi(x_{0}) -(p-2)\lambda_{\max}(D^{2}\varphi(x_{0}))\geq 0, \ \ 1<p<2.
\end{equation}
In the following, we only consider the proof of \eqref{loc:eq31}, since the proof of \eqref{loc:eq32} is analogous. We assume by contradiction that
\begin{equation}\label{loc:eq33}
-\Delta \varphi(x_{0}) -(p-2)\lambda_{\min}(D^{2}\varphi(x_{0}))< 0, \ \ p\geq 2,
\end{equation}
which implies that $ A $ has at least one positive eigenvalue. Otherwise, we see that $ \text{Tr}A +(p-2)\lambda_{\min}(A)\leq0 $, which is a contradiction to \eqref{loc:eq33}. Now we define $ \mathbb{R}^{n} = T \oplus S $, where $ T =\text{span}\{ e_{1},e_{2},\cdots,e_{k}\}$ is a linear subspace composed of those eigenvectors corresponding to nonnegative eigenvalues of $ A $. For $ \delta >0 $, we let
\begin{equation*}
\label{eq4.2998}
\Phi(x)= \varphi(x) + \delta |P_{T}(x-x_{0})|,
\end{equation*}
where $ P_{T}(x) $ denotes the orthogonal projection over $ T $. Noticing that $ u_{j} \rightarrow u_{\infty} $ locally uniformly in $ B_{1} $ and $ u_{\infty}-\varphi   $ attains its local minimum at $ x_{0} \in B_{1}    $, then for $ \delta >0 $ sufficiently small, $ u_{j}-\Phi $ attains a local minimum at some $ x_{j}^{\delta} \in B_{r}(x_{0}), 0<r<1 $. This allows us to select a subsequence $ x_{j}^{\delta} $ such that $ x_{j}^{\delta} \rightarrow \widehat{x} $ for some $ \widehat{x} \in B_{r}(x_{0}) $. Now we treat separately the cases where $ |P_{T}(x_{j}^{\delta}-x_{0})|=0 $ and $ |P_{T}(x_{j}^{\delta}-x_{0})|>0 $.

Case 2a.1: $ |P_{T}(x_{j}^{\delta}-x_{0})|=0 $. By the definition of $ |P_{T}(x_{j}^{\delta}-x_{0})| $, we have
\begin{equation*}
\label{eq4.2999}
\Phi(x)= \varphi(x) + \delta e \cdot P_{T}(x-x_{0}),
\quad \forall e \in \mathbb{S}^{n-1}.
\end{equation*}
Then direct calculation yields that
\begin{equation}\label{loc:eq34}D\Phi(x_{j}^{\delta})=A(x_{j}^{\delta}-x_{0})+\vec{b}+\delta P_{T}(e) \ \ \text{and} \ \ D^{2}\Phi(x_{j}^{\delta})=A.
\end{equation}
Here we choose $ e \in \mathbb{S}^{n-1}(|e|=1) $ such that $ P_{T}(e)=e $ provided that  $ e \in \mathbb{S}^{n-1} \cap T $, otherwise, $ P_{T}(e) = 0 $ provided $ e \in \mathbb{S}^{n-1} \cap T^{\bot} $, where $ T^{\bot} $ is the subspace orthogonal to $ T $. First, we consider

Case 2a.1.1: $ A(\widehat{x}-x_{0}) =0 $. For $ j $ large enough, we have
\begin{equation*}
|A(x_{j}^{\delta}-x_{0})+\vec{b}+\xi_{j}| \leq \frac{1}{4}\delta,
\end{equation*}
which, using the triangle inequality, reaches
\begin{equation}\label{loc:eq35}
|A(x_{j}^{\delta}-x_{0})+\vec{b}+\xi_{j}+\delta e| \geq \frac{3}{4}\delta.
\end{equation}
For simplicity, we denote
\begin{equation}\label{loc:eq36}
\left\{
     \begin{aligned}
     &  K_{j}=\frac{A(x_{j}^{\delta}-x_{0})+\vec{b}+\delta e+\xi_{j}}{|A(x_{j}^{\delta}-x_{0})+\vec{b}+\delta e+\xi_{j}|},        \\
     & K=\frac{A(\widehat{x}-x_{0})+\vec{b}+\delta e+\xi_{\infty}}{|A(\widehat{x}-x_{0})+\vec{b}+\delta e+\xi_{\infty}|},         \\
     \end{aligned}
     \right.
\end{equation}
then it is easy to see that $ K_{j} \rightarrow K  $ as $ j \rightarrow +\infty $, which, together with \eqref{loc:eq34} and \eqref{loc:eq36}, leads to
\begin{align}\label{loc:eq37}
\begin{split}
-\Delta_{p,\xi_{j}}^{\rm{N}}\Phi(x_{j}^{\delta})&= -\text{Tr}A-(p-2)\bigg\langle A \frac{K_{j}}{|K_{j}|}, \frac{K_{j}}{|K_{j}|}   \bigg\rangle     \\
 & \leq -\text{Tr}A -(p-2)\lambda_{\min}(A)\overset{\eqref{loc:eq33}}{<}0,
 \end{split}
\end{align}
where we have used the fact
\begin{equation*}
\text{Tr}\bigg( \left(\frac{K}{|K|}\otimes \frac{K}{|K|}\right)A \bigg) \geq \lambda_{\min}(A).
\end{equation*}
We combine \eqref{loc:eq22} and \eqref{loc:eq35} to read
\begin{equation*}
\label{eq4.304}
-\Delta_{p,\xi_{j}}^{\rm{N}}\Phi(x_{j}^{\delta}) \geq -\frac{j^{-1}}{\min\big\{1, \big(\frac{3\delta}{4}\big)^{a_{1}}, \big(\frac{3\delta}{4}\big)^{a_{2}}\big\}},
\end{equation*}
which, passing to the limit $ j \rightarrow +\infty $, leads to a contradiction with \eqref{loc:eq37}. Next we consider

Case 2a.1.2: $ A(\widehat{x}-x_{0}) \neq 0 $. For $ j $ large enough, we have
\begin{equation}\label{loc:eq38}
|A(x_{j}^{\delta}-x_{0})|\geq \frac{1}{4}|A(\widehat{x}-x_{0})|.
\end{equation}
In view of $ \vec{b}+\xi_{\infty} =0, |\vec{b}|, |\xi_{\infty}|>0 $ and $ \xi_{j} \rightarrow \xi_{\infty}  $, we obtain
\begin{equation}\label{loc:eq39}
|\vec{b}+\xi_{j}| \leq \frac{1}{16}|A(\widehat{x}-x_{0})|.
\end{equation}
Choosing $ \delta $ sufficiently small, we get $ \delta \leq \frac{1}{16}|A(\widehat{x}-x_{0})| $. This together with \eqref{loc:eq38} and \eqref{loc:eq39} can deduce
\begin{equation*}
|A(x_{j}^{\delta}-x_{0})+\vec{b}+\xi_{j}+\delta e| \geq \frac{1}{8}|A(\widehat{x}-x_{0})|.
\end{equation*}
As the discussions in Case 2a.1.1 above, it holds
\begin{equation*}
\label{eq4.308}
-\Delta_{p,\xi_{j}}^{\rm{N}}\Phi(x_{j}^{\delta}) \geq -\frac{j^{-1}}{\min\big\{1, \big(\frac{1}{8}|A(\widehat{x}-x_{0})|\big)^{a_{1}}, \big(\frac{1}{8}|A(\widehat{x}-x_{0})|\big)^{a_{2}}\big\}},
\end{equation*}
which, passing to the limit $ j \rightarrow +\infty $, also reaches a contradiction with \eqref{loc:eq37}.

Case 2a.2: $ |P_{T}(x_{j}^{\delta}-x_{0})|>0 $. Simple computation yields that
\begin{equation*}
\left\{
\begin{aligned}
     & D\Phi(x_{j}^{\delta})=A(x_{j}^{\delta}-x_{0})+\vec{b}+\delta D\big(|P_{T}(x_{j}^{\delta}-x_{0})|\big),       \\
     &  D^{2}\Phi(x_{j}^{\delta}) =A+\delta D^{2}\big(|P_{T}(x_{j}^{\delta}-x_{0})|\big) .        \\
     \end{aligned}
     \right.
\end{equation*}
For simplicity, we denote
\[\left\{ {\begin{array}{*{20}{c}}
B_{j}=D^{2}\big(|P_{T}(x_{j}^{\delta}-x_{0})|\big),\\
B= D^{2}\big(|P_{T}(\widehat{x}-x_{0})|\big)  ,
\end{array}} \right. \  \text{and} \
\left\{ {\begin{array}{*{20}{c}}
 E_{j}=A(x_{j}^{\delta}-x_{0})+\vec{b}+\delta D\big(|P_{T}(x_{j}^{\delta}-x_{0})|\big)+ \xi_{j}, \\
E=A(\widehat{x}-x_{0})+\vec{b}+\delta D\big(|P_{T}(\widehat{x}-x_{0})|\big)+ \xi_{j},
\end{array}} \right.\]
then it is not difficult to see that $ B_{j} \rightarrow B  $ and $ E_{j} \rightarrow E  $ as $ j \rightarrow +\infty $, respectively. Now we can rewrite
\begin{align}\label{loc:eq40}
\begin{split}
&-\Delta_{p,\xi_{j}}^{\rm{N}}\Phi(x_{j}^{\delta})    \\
&=\underbrace{-\text{Tr}A-(p-2)\bigg \langle A\frac{E_{j}}{|E_{j}|}, \frac{E_{j}}{|E_{j}|}\bigg \rangle }_{:=\Sigma_{1}} -\underbrace{\delta \bigg\{ \text{Tr}B_{j}+(p-2)\bigg \langle B_{j}\frac{E_{j}}{|E_{j}|}, \frac{E_{j}}{|E_{j}|}\bigg \rangle\bigg\}}_{:=\Sigma_{2}}.
\end{split}
\end{align}
Noting that $ \text{Tr}\bigg( \left(\frac{E}{|E|}\otimes \frac{E}{|E|}\right)A \bigg) \geq \lambda_{\min}(A) $, then we see the term $ \Sigma_{1} <0 $. Additionally, noting that $ |P_{T}(x-x_{0})| $ is smooth and convex in a small neighbourhood of $ x_{j}^{\delta} $, which implies that $ B_{j} $ is a nonnegative definite, thereby leading to $ \Sigma_{2} \geq 0 $. Now we conclude that $-\Delta_{p,\xi_{j}}^{\rm{N}}\Phi(x_{j}^{\delta})<0 $. Thereafter, we can follow the process of Case 2a.1.1 and Case 2a.1.2 to derive a contradiction to \eqref{loc:eq37}.

For Case $ 2b $, its proof is very analogous to the Case 2a. We skip the details.

As has been stated above, we have that $ u_{\infty} $ is a viscosity solution to \eqref{loc:eq29}. Using Lemma \ref{lem4.21}, we conclude that $ u_{\infty} \in C^{1,\beta_{1}}_{loc}(B_{1}) $ for some $ \beta_{1} \in (0,1) $. Now setting $ \widehat{\beta} =\min \big\{\widetilde{\alpha},\beta_{1}\big\}$, then we derive $ u_{\infty} \in C^{1,\widehat{\beta}}_{loc}(B_{1}) $, which means that for any $ \rho \in (0,1) $, there exists $ \vec{q'} \in \mathbb{R}^{n}    $ such that
\begin{equation*}
 \mathop {osc}\limits_{B_{\rho}}\big(u_{\infty}-\vec{q'}\cdot x\big) \leq C\rho^{1+\widehat{\beta}}.
\end{equation*}
Choosing $ \rho $ such that $ C\rho^{\widehat{\beta}} \leq \frac{1}{2} $, then we have
\begin{equation*}
 \mathop {osc}_{B_{\rho}}\big(u_{\infty}-\vec{q'}\cdot x\big) \leq    \frac{1}{2}\rho,
\end{equation*}
which contradicts to \eqref{loc:eq23}.
\end{proof}

\subsection{Iteration}\label{Section 4.3}
Based on the improvement of flatness lemma, we shall establish an oscillation control at discrete scales.
\begin{Lemma}
\label{lem4.31}
Suppose that $ u $ is a viscosity sub-solution to \eqref{1.10} and $ u $ is a viscosity super-solution to \eqref{1.11}. There exist $ 0 < \rho \ll 1 $ and $ \delta > 0 $, which are the same as the Lemma \ref{lem4.22}, such that if $ ||u||_{L^{\infty}(B_{1})} \leq \frac{1}{2}$ and $ ||f||_{L^{\infty}(B_{1})} \leq \delta   $, then there exists $ \alpha' \in (0,\frac{1}{1+a_{2}}]   $ such that for all $ j \in \mathbb{R} $, there exists $ \xi_{j} \in  \mathbb{R}^{n}  $ such that
\begin{equation*}
\label{eq4.180}
\mathop {osc}\limits_{B_{\rho^{j}}} (u(x)-\xi_{j}\cdot x) \leq \rho^{j(1+\alpha')}.
\end{equation*}
\end{Lemma}
\begin{proof}
We argue by induction. Firstly, for the case $ j=0$, it is  obvious that the conclusion is true by $ ||u||_{L^{\infty}(B_{1})} \leq \frac{1}{2}$. Suppose the conclusion holds for the case $ j=k $, now we are going to consider the case $ j=k+1 $.

Let
\begin{equation*}
\label{eq4.181}
u_{k}(x) = \frac{u(\rho ^{k}x)-\xi_{k}\cdot(\rho ^{k}x)}{\rho ^{k(1+\alpha')}},
\end{equation*}
then direct calculation yields that $ u_{k} $ is a viscosity sub-solution to
\begin{equation*}
\label{eq4.191}
\min \bigg \{-\Delta_{p,\overline{\xi}_{k}}^{\rm{N}}u_{k}, -H_{k}(Du_{k},\overline{\xi}_{k};a_{1})\Delta_{p,\overline{\xi}_{k}}^{\rm{N}}u_{k}, -H_{k}(Du_{k},\overline{\xi}_{k};a_{2})\Delta_{p,\overline{\xi}_{k}}^{\rm{N}}u_{k}  \bigg \} = ||\widetilde{f}||_{L^{\infty}(B_{1})},
\end{equation*}
and $ u_{k} $ is a viscosity super-solution to
\begin{equation*}
\label{eq4.201}
\max \bigg \{-\Delta_{p,\overline{\xi}_{k}}^{\rm{N}}u_{k}, -H_{k}(Du_{k},\overline{\xi}_{k};a_{1})\Delta_{p,\overline{\xi}_{k}}^{\rm{N}}u_{k}, -H_{k}(Du_{k},\overline{\xi}_{k};a_{2})\Delta_{p,\overline{\xi}_{k}}^{\rm{N}}u_{k}  \bigg \} = -||\widetilde{f}||_{L^{\infty}(B_{1})},
\end{equation*}
where
\begin{equation*}
\label{eq4.17}
\left\{
     \begin{aligned}
     & H_{k}(Du_{k},\overline{\xi}_{k};a_{i}): = |Du_{k}+\overline{\xi}_{k}|^{a_{i}},\ i=1,2   ,   \\
     &  \overline{\xi}_{k}:=\rho^{-k\alpha'}\xi_{k},     \\
     &   ||\widetilde{f}||_{L^{\infty}(B_{1})} := \max\big\{1,\rho^{-ka_{1}\alpha'},\rho^{-ka_{2}\alpha'}\big\}\rho^{k(1-\alpha')}||f||_{L^{\infty}(B_{1})}  .    \\
     \end{aligned}
     \right.
\end{equation*}
Noticing that $ 0 < \rho \ll 1 $ and $ 0<a_{1} \leq a_{2} <\infty$, we infer
$$ ||\widetilde{f}||_{L^{\infty}(B_{1})} \leq \rho ^{k[1-(1+a_{2})\alpha']}||f||_{L^{\infty}(B_{1})} \leq \delta ,  $$
where $ 0< \alpha' \leq\frac{1}{1+a_{2}}$ is used.

Applying Lemma \ref{lem4.22} to $ u_{k} $, then there exists $ \vec{q}_{k+1} \in  \mathbb{R}^{n}  $ such that
$$ \mathop {osc}\limits_{B_{\rho}}(u_{k}-\vec{q}_{k+1} \cdot x) \leq \frac{1}{2}\rho ,    $$
which means that
\begin{equation}\label{loc:eq41}
 \mathop   {osc}\limits_{B_{\rho^{k+1}}}\big(u(x)-\xi_{k+1} \cdot x\big) \leq \frac{1}{2}\rho^{1+k(1+\alpha')} ,
\end{equation}
where
\begin{equation*}
\xi_{k+1} =\xi_{k}+\vec{q}_{k+1}\rho^{k\alpha'}.
\end{equation*}
In light of $ \alpha' \in (0,\frac{1}{1+a_{2}}]   $, we can take $ \alpha' $ such that
\begin{equation}\label{loc:eq42}
 \frac{1}{2} \leq \rho^{\alpha'}.
\end{equation}
Finally, we combine \eqref{loc:eq41} and \eqref{loc:eq42} to end the induction.
\end{proof}

Once Lemma \ref{lem4.31} is available, we can finish the proof of Theorem \ref{pro2} in a straightforward way. In effect, whenever $ r \in (0,1] $, there exists $ j \in \mathbb{N} \cup \{0\} $ such that $ \rho^{j+1} < r \leq \rho^{j}   $. Thus we have
\begin{align*}
\mathop {osc}_{B_{r}} (u-\xi_{j}\cdot x) \leq  \mathop {osc}_{B_{\rho^{j}}} (u-\xi_{j}\cdot x) \leq \rho^{j(1+\alpha')}& = \rho^{(j+1)(1+\alpha')} \rho^{-(1+\alpha')} \leq \rho^{-(1+\alpha')} r^{1+\alpha'}    \\
& = C(n,p,a_{1},a_{2}) r^{1+\alpha'},
\end{align*}
which implies that $ u $ is $ C^{1,\alpha'}(0) $. Then by a standard translation argument to conclude the proof of Theorem \ref{pro2}.

\subsection{A corollary}\label{Section 4.4}
Finally, we conclude this section by giving a direct consequence of Theorem \ref{pro2}, which will be frequently used in the proof of Theorem \ref{thm2}.

\begin{Corollary}\label{Section4:Coro1}
 Suppose $ u\in C(B_{1}) $ is a viscosity sub-solution to
$$ \min_{i=0,1,\cdots,M} \bigg \{-\big(|Du|^{\alpha_{i}(x)} + a(x)|Du|^{\beta_{i}(x)} \big)   \Delta_{p}^{\rm{N}} u \bigg\}  = ||f||_{L^{\infty}(B_{1})} \ \ \text{in} \ \ B_{1}      $$
and is also a viscosity super-solution to
$$ \max_{i=0,1,\cdots,M} \bigg \{-\big(|Du|^{\alpha_{i}(x)} + a(x)|Du|^{\beta_{i}(x)} \big)\Delta_{p}^{\rm{N}} u \bigg\} = -||f||_{L^{\infty}(B_{1})} \ \ \text{in} \ \ B_{1},  $$
where $ \alpha_{i}(x), \beta_{i}(x) $ are given in \eqref{intro:eq115}--\eqref{intro:eq118}, $ i=0,1,\cdots,M $. Then $ u \in C^{1,\alpha'}_{loc}(B_{1}) $ for some $ 0 < \alpha' <1 $.
\end{Corollary}

\begin{proof}
This proof is exactly identical to the proof of Theorem \ref{thm1} in Section \ref{Section 5}, using the standard argument.
\end{proof}

\section{Proof of Theorem \ref{thm1} }\label{Section 5}
    ~~~~In this section, we provide a simple proof of Theorem \ref{thm1} by using Theorem \ref{pro2}.
\begin{proof}[Proof of Theorem~\ref{thm1}]
Assume $ \varphi \in C^{2}(B_{1}) $ and $ u-\varphi $ has a local maximum at $ x_{0} \in B_{1} $, then by the definition of viscosity solution, we have
\begin{equation*}
\label{eq3.31}
-\bigg( |D\varphi(x_{0})|^{\alpha(x_{0},u(x_{0}))} +a(x_{0})|D\varphi(x_{0})|^{\beta(x_{0},u(x_{0}))}\bigg)\Delta_{p}^{\rm{N}} \varphi(x_{0}) \leq f(x_{0}).
\end{equation*}
Now we divide the discussion into two cases.

Case 1: $ \Delta_{p}^{\rm{N}} \varphi(x_{0}) \geq 0 $. It is obvious that
\begin{equation*}
\label{eq3.32}
-|D\varphi(x_{0})|^{a_{i}} \Delta_{p}^{\rm{N}} \varphi(x_{0}) \leq ||f||_{L^{\infty}(B_{1})}, \ \ i=1,2.
\end{equation*}

Case 2: $ \Delta_{p}^{\rm{N}} \varphi(x_{0}) \leq 0 $. By direct calculations, one of the following inequalities must hold
$$ -|D\varphi(x_{0})|^{a_{1}} \Delta_{p}^{\rm{N}} \varphi(x_{0}) \leq ||f||_{L^{\infty}(B_{1})}  ,  $$
$$ -|D\varphi(x_{0})|^{a_{2}} \Delta_{p}^{\rm{N}} \varphi(x_{0}) \leq ||f||_{L^{\infty}(B_{1})},$$
under either the condition $ |D\varphi(x_{0})| \geq 1  $ or the condition $ |D\varphi(x_{0})| \leq 1 $.

Combining the above two cases, we derive that $ u $ is a viscosity sub-solution to
$$ \min_{i=0,1,2} \bigg \{-|Du|^{a_{i}} \Delta_{p}^{\rm{N}}u \bigg \}=||f||_{L^{\infty}(B_{1})} .$$
Similarly, we see that $ u $ is a viscosity super-solution to
$$ \max_{i=0,1,2} \bigg \{ -|Du|^{a_{i}}\Delta_{p}^{\rm{N}}u \bigg \}=-||f||_{L^{\infty}(B_{1})}, $$
where $ a_{0} = 0 $. Using Theorem \ref{pro2}, we immediately obtain $ u \in C^{1,\alpha'}_{loc}(B_{1}) $  and the desired estimate \eqref{eq4}.
\end{proof}

\vspace{3mm}

\section{Proof of Theorem \ref{thm2}}\label{Section 6}
   ~~~~In this section, we shall give the proof of Theorem \ref{thm2}. In Section \ref{Section 6.1}, we reduce the proof to the case when $ ||u||_{L^{\infty}(B_{1})} \leq 1 $ and $ ||f||_{L^{\infty}(B_{1})} \leq \delta $ for some positive constant $ \delta $. In Section \ref{Section 6.2}, we present an approximation lemma. In Section \ref{Section 6.3}, an improvement of iteration is given. Finally, we conclude the proof of Theorem \ref{thm2} via Proposition \ref{prop6.3}.

\subsection{Reduction of the problem}\label{Section 6.1}
In fact, suppose $ u $ is a viscosity solution to \eqref{intro:eq1}, then it is not difficult to recognize that $ u $ is a viscosity sub-solution to
\begin{equation}\label{the:eq1}
\min_{i=0,1,\cdots,M} \bigg \{-\bigg(|Du|^{\alpha_{i}(x)} +    a(x)|Du|^{\beta_{i}(x)} \bigg)   \Delta_{p}^{\rm{N}} u \bigg\}  = ||f||_{L^{\infty}(B_{1})}
\end{equation}
and is also a viscosity super-solution to
\begin{equation}\label{the:eq2}
\max_{i=0,1,\cdots,M} \bigg \{-\bigg(|Du|^{\alpha_{i}(x)} + a(x)|Du|^{\beta_{i}(x)} \bigg)\Delta_{p}^{\rm{N}} u \bigg\} = -||f||_{L^{\infty}(B_{1})}.
\end{equation}
Now we will show that a simple scaling reduces the proof of the problem to the case that
\begin{equation}\label{the:eq3}
||u||_{L^{\infty}(B_{1})} \leq 1 \ \ \text{and} \ \ ||f||_{L^{\infty}(B_{1})} \leq  \delta
\end{equation}
for some positive constant $ \delta $ to be determined. More precisely, we have the following proposition.
\begin{Proposition}
\label{prop6.1}
Assume \eqref{the:eq1}--\eqref{the:eq3} hold and $u$ satisfies
$$  ||u||_{C^{1,\tau}(x_{0})} \leq C . $$
Then Theorem \ref{thm2} holds.
\begin{proof}
For fixed $ x_{0} \in B_{1} $, we let
\begin{equation*}
v(x) = \frac{u(x_{0}+\widehat{\tau} x)}{K},
\end{equation*}
where
$$ \widehat{\tau} = \delta^{\frac{1}{2+\min\limits_{i=0,1,\cdots,M}\inf\limits_{B_{1}}\alpha_{i}(x)}} \ \ \text{and} \ \ K=1+||u||_{L^{\infty}(B_{1})} +||f||_{L^{\infty}(B_{1})}^{\frac{1}{1+\min\limits_{i=0,1,\cdots,M}\inf\limits_{B_{1}}\alpha_{i}(x)}}  .  $$
Then a direct calculation yields that $ v $ is a viscosity sub-solution to
\begin{equation*}
\min_{i=0,1,\cdots,M} \bigg\{ -\bigg(|Dv|^{\widetilde{\alpha}_{i}(x)} +\widetilde{a}(x)|Dv|^{\widetilde{\beta_{i}}(x)}\bigg) \Delta_{p}^{\rm{N}}v\bigg\} = ||\widetilde{f}||_{L^{\infty}(B_{1})}
\end{equation*}
and it is also a viscosity super-solution to
\begin{equation*}
\max_{i=0,1,\cdots,M} \bigg\{ -\bigg(|Dv|^{\widetilde{\alpha}_{i}(x)} +\widetilde{a}(x)|Dv|^{\widetilde{\beta_{i}}(x)}\bigg) \Delta_{p}^{\rm{N}}v\bigg\} = -||\widetilde{f}||_{L^{\infty}(B_{1})},
\end{equation*}
where
\begin{equation*}
\left\{
     \begin{aligned}
     & \widetilde{\alpha_{i}}(x) :=\alpha_{i}(x_{0}+\widehat{\tau} x), \  \widetilde{\beta_{i}}(x):=\beta_{i}(x_{0}+ \widehat{\tau} x); \\
      & \widetilde{a}(x):= a(x_{0}+\widehat{\tau} x)\left(\frac{K}{\widehat{\tau}}\right)^{\widetilde{\beta_{i}}(x)-\widetilde{\alpha}_{i}(x)};                     \\
&  \widetilde{f}(x):= \max_{i=0,1,\cdots,M} \bigg\{  \frac{\widehat{\tau}^{2+\widetilde{\alpha}_{i}(x)}}{K^{1+\widetilde{\alpha}_{i}(x)}}\bigg\}||f||_{L^{\infty}(B_{1})}.  \\
     \end{aligned}
     \right.
\end{equation*}

Noticing that
\begin{equation*}
||v||_{L^{\infty}(B_{1})} \leq 1 \ \ \text{and} \ \ ||\widetilde{f}||_{L^{\infty}(B_{1})} \leq \delta
\end{equation*}
provided $ \delta $ is sufficiently small.

Applying our assumptions, it follows
$$ ||v||_{C^{1,\tau}(x_{0})} \leq C ,    $$
which implies that
$$ ||u||_{C^{1,\tau}(x_{0})} \leq C \left( 1+||u||_{L^{\infty}(B_{1})} +||f||_{L^{\infty}(B_{1})}^{\gamma}  \right),    $$
where
$$  \gamma = \frac{1}{1+\min\limits_{i=0,1,\cdots,M}\inf\limits_{B_{1}}\alpha_{i}(x)}.    $$
This completes the proof of Proposition \ref{prop6.1}.
\end{proof}
\end{Proposition}

\subsection{Approximation}\label{Section 6.2}
Here we shall use our regularity result from Corollary \ref{Section4:Coro1} to show that the solutions to \eqref{the:eq1} and \eqref{the:eq2} can be approximated by normalized $ p$-harmonic function in a $ C^{1}_{loc} $ fashion.
\begin{Proposition}\label{prop6.2}
Assume \eqref{the:eq1}--\eqref{the:eq2} hold. For given $ 0< \epsilon <1 $, there exists $ \delta >0 $, depending only on $ n, p, a_{1}, a_{2}$, such that if $ ||u||_{L^{\infty}(B_{1})} \leq 1    $ with $ ||f||_{L^{\infty}(B_{1})} \leq \delta   $, then there exists a function $ v $ satisfying
$$ -\Delta_{p}^{\rm{N}} v=0 \ \ \text{in} \ \ B_{r}\ (0<r<1)  $$
in the viscosity sense, and
\begin{equation*}
\label{eq6.22}
 ||u-v||_{L^{\infty}(B_{r})} \leq \epsilon \ \ \text{and} \ \ ||Du-Dv||_{L^{\infty}(B_{r})} \leq \epsilon .
\end{equation*}
\begin{proof}
We prove by contradiction. Suppose that there exist $ \{u_{k}\}_{k}$, $\{f_{k}\}_{k}$,$ \{\alpha_{i}^{k}\}_{k}$, $\{\beta_{i}^{k}\}_{k} $,$ \{a_{k}\}_{k}$ and $ \epsilon_{0} \in (0,1)    $ such that the following conclusions hold:

(A1). $ ||u_{k}||_{L^{\infty}(B_{1})} \leq 1 $ and $ ||f_{k}||_{L^{\infty}(B_{1})} \leq \frac{1}{k}    $ ;

(A2). $ 0<a_{1}\leq \alpha_{i}^{k}(x) \leq \beta_{i}^{k}(x)\leq a_{2} <+\infty $, $ i=0,1,\cdots,M $ ;

(A3). For any function $ v $ satisfing
$$ -\Delta_{p}^{\rm{N}} v=0 \ \ \text{in} \ \ B_{r}  $$
in the viscosity sense, we have
\begin{equation}\label{the:eq4}
 \max \bigg \{ ||u_{k}-v||_{L^{\infty}(B_{r})}, ||Du_{k}-Dv||_{L^{\infty}(B_{r})}\bigg \} > \epsilon_{0} .
\end{equation}
By definition, we have that $ \{u_{k}\}_{k} $ is a viscosity sub-solution to
\begin{equation*}
 \min_{i=0,1,\cdots,M} \bigg \{-\bigg(|Du_{k}|^{\alpha_{i}^{k}(x)} + a_{k}(x)|Du_{k}|^{\beta_{i}^{k}(x)} \bigg)   \Delta_{p}^{\rm{N}} u_{k} \bigg\}  = ||f_{k}||_{L^{\infty}(B_{1})}
\end{equation*}
and it is also a viscosity super-solution to
\begin{equation*}
\max_{i=0,1,\cdots,M} \bigg \{-\bigg(|Du_{k}|^{\alpha_{i}^{k}(x)} + a_{k}(x)|Du_{k}|^{\beta_{i}^{k}(x)} \bigg)\Delta_{p}^{\rm{N}} u_{k} \bigg\} = -||f_{k}||_{L^{\infty}(B_{1})},
\end{equation*}
which, together with (A1) and (A2), yield that $ u_{k} \in C^{1,   \alpha'}_{loc}(B_{1})   $ for some $ \alpha' \in (0,1)$, where Corollary \ref{Section4:Coro1} is used. Then from the  Arzel$ \grave{a}$-Ascoli theorem, it follows that
\begin{equation}\label{the:eq5}
 u_{k} \rightarrow u_{\infty} \ \text{uniformly in} \ \  B_{r} \ \ \text{and} \ \   Du_{k} \rightarrow Du_{\infty}\ \  \text{uniformly in} \ \  B_{r}.
\end{equation}
In addition, from Lemma \ref{lem2.2} and (A1), the limit function $ u_{\infty} $ is a viscosity solution to
\begin{equation}\label{6.666}
-\Delta_{p}^{\rm{N}} u_{\infty}=0 \ \ \text{in} \ \ B_{r}.
\end{equation}
Then \eqref{the:eq4}, \eqref{the:eq5} and \eqref{6.666} lead to a contradiction.
\end{proof}
\end{Proposition}

\subsection{Improved oscillation-type estimate}\label{Section 6.3}
Thanks to Proposition \ref{prop6.2}, we can derive an improved oscillation-type estimate of solution $ u $ to \eqref{the:eq1} and \eqref{the:eq2} near the set $ \{ x: Du(x)=0\} $. The proof makes use of some ideas from \cite[Lemma 2.5]{JG20} and \cite[Theorem 5.4]{AME17}.

First we establish the first step of the geometric control on the growth of the gradient.

\begin{Lemma}\label{Sec6:lem61}
Under the assumptions in Proposition \ref{prop6.2}, and for every $ 0< \tau' < \widehat{\beta_{0}} $, then there exists universal constant $ 0< \rho \ll 1 $ such that
\begin{equation*}
\label{eq6.32}
\sup_{B_{\rho}(0)} |u(x)-u(0)| \leq \rho^{1+\tau'} +|Du(0)|\rho .
\end{equation*}
\end{Lemma}
\begin{proof}
Let $ \epsilon_{0} >0 $ be determined. From Proposition \ref{prop6.2}, there exists $ \delta_{\epsilon_{0}}>0    $ such that there exists a normalized $ p$-harmonic function $ v $ satisfying
\begin{equation}
\label{App11}
 ||u-v||_{L^{\infty}(B_{r})} \leq \epsilon_{0} \ \ \text{and} \ \ ||Du-Dv||_{L^{\infty}(B_{r})} \leq \epsilon_{0}
\end{equation}
provided $ ||f||_{L^{\infty}(B_{1})} \leq  \delta_{\epsilon_{0}}     $ holds.

Now by choosing
\begin{equation}\label{App12}
0< \epsilon_{0} \leq \frac{1}{6} \rho^{1+\tau'}  \ \ \text{and} \ \ 0<\rho \leq \min \bigg \{r, \left(  \frac{1}{2C}\right)^{\frac{1}{\widehat{\beta_{0}}-\tau'}}  \bigg \} ,
 \end{equation}
then from triangle inequality and (\ref{App11}), it follows that
\begin{align}\label{the:eq23}
\begin{split}
\sup_{B_{\rho}(0)}|u(x)-u(0)-Du(0)\cdot x| \leq &  \sup_{B_{\rho}(0)} |u(x)-v(x)| + \sup_{B_{\rho}(0)}|v(x)-v(0)-Dv(0)\cdot x|      \\
& +\sup_{B_{\rho}(0)}|v(0)-u(0)+(Dv(0)-Du(0))\cdot x|      \\
\leq &  \frac{1}{2} \rho^{1+\tau'} + \sup_{B_{\rho}(0)}|v(x)-v(0)-Dv(0)\cdot x|.
\end{split}
\end{align}
By virtue of Lemma \ref{lem2.3}, we have that $ v \in C^{1,\widehat{\beta_{0}}}_{loc}(B_{1}) $ with the estimate
\begin{equation*}
\sup_{B_{\rho}(0)}|v(x)-v(0)-Dv(0)\cdot x| \leq C \rho^{1+\widehat{\beta_{0}}},
\end{equation*}
which, together with \eqref{App12} and \eqref{the:eq23}, yields that
$$\sup_{B_{\rho}(0)}|u(x)-u(0)-Du(0)\cdot x| \leq \rho^{1+\tau'}. $$
The proof is now complete.
\end{proof}

Next we iterate the previous estimate to control the oscillation of the solutions in $ \rho$-adic balls.

\begin{Proposition}\label{prop6.3}
Under the assumptions in Proposition \ref{prop6.2}, there exists a non-decreasing sequence $ \{ \tau_{k}\} $ and a universal constant $ 0< \rho \ll 1$ such that
\begin{equation}\label{the:eq6}
  \sup_{B_{\rho^{k}}(0)} |u(x)-u(0)| \leq \rho^{k(1+\tau_{k})} +|Du(0)| \sum_{i=0}^{k-1} \rho^{k+i\tau_{k}}.
\end{equation}
\end{Proposition}
\begin{proof}
The desired non-decreasing sequence is defined as follows
$$ \tau_{k} := \min_{i=0,1,\cdots,M}\bigg \{\widehat{\beta_{0}}^{-}, \min_{B_{\rho^{k}(0)}} \frac{1}{1+\alpha_{i}(x)}    \bigg \},        $$
which converges to
$$ \tau:= \min_{i=0,1,\cdots,M}\bigg \{\widehat{\beta_{0}}^{-}, \frac{1}{1+\alpha_{i}(0)}    \bigg \} .  $$
Now we argue by finite induction. The case $ k =1 $ is clearly the statement of Lemma \ref{Sec6:lem61}. Suppose we have verified \eqref{the:eq6} for $ j=1,2,\cdots,k $. Letting
\begin{equation*}
\label{eq6.33}
v_{k}(x) = \frac{u(\rho^{k}x)-u(0)}{A_{k}} ,
\end{equation*}
where $ A_{k}:= \rho^{k(1+\tau_{k})} +|Du(0)|\sum_{i=0}^{k-1} \rho^{k+i\tau_{k}}$, a simple calculation yields that $ \{v_{k}\}_{k} $ is a viscosity sub-solution to
\begin{equation*}
\label{eq6.34}
\min_{i=0,1,\cdots,M} \bigg \{ -\bigg(|Dv_{k}|^{\widetilde{\alpha}_{i}(x)} + \widetilde{a}(x)|Dv_{k}|^{\widetilde{\beta_{i}}(x)}    \bigg)\Delta_{p}^{\rm{N}}v_{k}\bigg \}= ||\widetilde{f}||_{L^{\infty}(B_{1})}
\end{equation*}
and is also a viscosity super-solution to
\begin{equation*}
\label{eq6.35}
\max_{i=0,1,\cdots,M} \bigg \{ -\bigg(|Dv_{k}|^{\widetilde{\alpha}_{i}(x)} + \widetilde{a}(x)|Dv_{k}|^{\widetilde{\beta_{i}}(x)}    \bigg)\Delta_{p}^{\rm{N}}v_{k}\bigg \} = - ||\widetilde{f}||_{L^{\infty}(B_{1})},
\end{equation*}
where
$$ \widetilde{f}(x):= \max_{i=0,1,\cdots,M}\bigg\{ \frac{\rho^{k(2+\widetilde{\alpha}_{i}(x))}}{A_{k}^{1+\widetilde{\alpha}_{i}(x)}}\bigg\}||f||_{L^{\infty}(B_{1})}, \ \ \widetilde{a}(x):= a(\rho^{k}x)\left( \frac{A_{k}}{\rho^{k}}\right)^{\widetilde{\beta}_{i}(x)-\widetilde{\alpha}_{i}(x)}     $$
and
$$ \widetilde{\alpha}_{i}(x) := \alpha_{i}(\rho^{k}x), \ \ \widetilde{\beta}_{i}(x) := \beta_{i}(\rho^{k}x).      $$
We observe that $ ||v_{k}||_{L^{\infty}(B_{1})} \leq 1 $ and
\begin{equation*}
\label{eq6.36}
||\widetilde{f}||_{L^{\infty}(B_{1})} \leq \max_{i=0,1,\cdots,M} \big\{  \rho^{k[1-\tau_{k}(1+\alpha_{i}(\rho^{k}x))]}\big\}||f||_{L^{\infty}(B_{1})} \leq  ||f||_{L^{\infty}(B_{1})} \leq \delta,
\end{equation*}
where the definition of $ \tau_{k} $ and inductive hypothesis are used. Then from Lemma \ref{Sec6:lem61}, we have that
\begin{equation*}
\sup_{B_{\rho}(0)}|v_{k}(x)-v_{k}(0)| \leq \rho^{1+\tau'} +|Dv_{k}(0)|\rho ,
\end{equation*}
which implies that
\begin{align}\label{the:eq7}
\begin{split}
& \sup_{B_{\rho^{k+1}}(0)}|u(x)-u(0)| \leq \rho^{1+\tau'}A_{k} +\rho^{k+1}|Du(0)|      \\
 & \ \ \ \ \ \ \ \ \  \ \ \ \ \ \ \ \ \ \ \ \ \ \ \ \ \ = \underbrace{\rho^{k(1+\tau_{k})+1+\tau'}}_{:=I_{1}}+\underbrace{|Du(0)|\left(\rho^{1+k}+\rho^{1+\tau'}\sum_{i=0}^{k-1}\rho^{k+i\tau_{k}}\right)}_{:=I_{2}}.
\end{split}
\end{align}

We first derive the estimate of $I_{1}$. Using \eqref{intro:eq117}, it follows that
\begin{align*}
 k\left(\frac{1}{1+\alpha_{i}(0)}-\frac{1}{1+\max_{B_{\rho^{k}(0)}}\alpha_{i}(x)}\right) & \leq k \left(\max_{B_{\rho^{k}(0)}}\alpha_{i}(x) - \alpha_{i}(0)\right)       \nonumber  \\
 &  \leq k \omega(\rho^{k}) ,
\end{align*}
which means that
\begin{equation}\label{the:eq8}
0 \leq k(\tau-\tau_{k}) \leq k\omega(\rho^{k}).
\end{equation}
This, together with Remark \ref{rk3}, yields that
\begin{equation}\label{the:eq9}
\rho^{k(\tau_{k}-\tau_{k+1})} \leq \rho^{-k\omega(\rho^{k})} \leq \rho^{\frac{\tau-\widehat{\beta_{0}}}{2}}.
\end{equation}
In view of the arbitraries of $ \tau' $, we set $ \tau'  = \frac{\tau+\widehat{\beta_{0}}}{2} < \widehat{\beta_{0}}$. Then from \eqref{the:eq9}, we get
\begin{align}\label{the:eq10}
\begin{split}
  I_{1}& \leq \rho^{(k+1)(1+\tau_{k+1})} \rho^{k(\tau_{k}-\tau_{k+1})+(\tau'-\tau_{k+1})} \\
 & \leq \rho^{(k+1)(1+\tau_{k+1})} \rho^{\frac{\tau-\widehat{\beta_{0}}}{2}+\tau'-\tau_{k+1}}       \\
 & \leq \rho^{(k+1)(1+\tau_{k+1})}\rho^{\tau-\tau_{k+1}} \leq  \rho^{(k+1)(1+\tau_{k+1})} .
 \end{split}
\end{align}

We then derive the estimate of  $I_{2}$. It suffices to show that
\begin{equation}\label{the:eq11}
\rho^{1+k}+\rho^{1+\tau'}\sum_{i=0}^{k-1}\rho^{k+i\tau_{k}} \leq \sum_{i=0}^{k} \rho^{1+k+i\tau_{k+1}},
\end{equation}
where $ \tau'  = \frac{\tau+\widehat{\beta_{0}}}{2} < \widehat{\beta_{0}} $.

Before proving \eqref{the:eq11}, we make the following claim.

$ \textbf{Claim:}$ $ (j-1)\tau_{j} +\tau' \geq j \tau_{j+1} , j=1,2,\cdots,k  .      $

Indeed, from \eqref{the:eq8} and Remark \ref{rk3}, we have that
\begin{equation}\label{the:eq12}
j(\tau_{j+1}-\tau_{j}) \leq \frac{\widehat{\beta_{0}}-\tau}{2} +j(\tau_{j+1}-\tau), \ \  j=1,2,\cdots,k.
\end{equation}
We observe that
\begin{equation}\label{the:eq13}
\frac{\widehat{\beta_{0}}-\tau}{2}+j(\tau_{j+1}-\tau)+\tau_{j}-\tau' \leq 0,
\end{equation}
where $ \tau' $ and $ \tau_{j} \nearrow \tau $ are used. Now we combine \eqref{the:eq12} and \eqref{the:eq13} to deduce
$$ j(\tau_{j+1}-\tau_{j})+\tau_{j}-\tau' \leq 0,  $$
which implies the claim holds.

Recalling that $ \rho \in (0,1)$ and using Claim above, then \eqref{the:eq11} is true. Thus, we derive
\begin{equation}\label{the:eq14}
I_{2} \leq |Du(0)|\sum_{i=0}^{k} \rho^{1+k+i\tau_{k+1}}.
\end{equation}

Combining \eqref{the:eq7}, \eqref{the:eq10} and \eqref{the:eq14}, we end the induction and complete the proof of this proposition.
\end{proof}

With the help of Proposition \ref{prop6.3}, we now proceed with the proof of Theorem \ref{thm2}.

\begin{proof}[Proof of Theorem~\ref{thm2}]

For simplicity, we assume $ x_{0} =0 $. Then our proof is divided into two cases: $|Du(0)|$ is sufficiently small or not.

$\text{Case 1}$. If $  |Du(0)| \leq r^{\tau}  $ for some fixed small $ 0<r<1 $, then there exists $ k \in \mathbb{N} $ such that $ \rho^{k+1} \leq r < \rho^{k}  $. Using Proposition \ref{prop6.3}, we obtain
\begin{align*}
\sup_{B_{r}(0)} |u(x)-u(0)-& Du(0)\cdot x|  \leq \sup_{B_{\rho^{k}}(0)}  |u(x)-u(0)-Du(0)\cdot x|                  \\
 &  \leq  \sup_{B_{\rho^{k}}(0)} |u(x)-u(0)| + |Du(0)|\rho^{k}    \\
 & \overset{\eqref{the:eq6}}{\leq} \rho^{k(1+\tau_{k})} +|Du(0)| \sum_{i=0}^{k-1} \rho^{k+i\tau_{k}}+ |Du(0)|\rho^{k}    \\
 &  \leq   \rho^{k(1+\tau_{k})} \left(1+ |Du(0)|\rho^{-k(1+\tau_{k})}\sum_{i=0}^{k-1} \rho^{k+i\tau_{k}}    \right)  + |Du(0)|\rho^{k}    \\
 &  \leq \rho^{k(\tau_{k}-\tau)}\rho^{k(1+\tau)}\left(1+ r^{\tau-\tau_{k}}\sum_{i=0}^{k-1} \rho^{i\tau_{k}}    \right)  + \frac{r^{1+\tau}}{\rho}   \\
 &   \leq  \rho^{k(\tau_{k}-\tau)} \frac{r^{1+\tau}}{\rho^{1+\tau}}\left(1+ r^{\tau-\tau_{k}}\frac{1}{1-\rho^{\tau_{k}}}\right)+ \frac{r^{1+\tau}}{\rho}   \\
 &  \leq \rho^{k(\tau_{k}-\tau)} \frac{r^{1+\tau}}{\rho^{1+\tau}(1-\rho^{\tau_{k}})}\left(1+r^{\tau-\tau_{k}}\right) + \frac{r^{1+\tau}}{\rho}   \\
 & \leq \rho^{k(\tau_{k}-\tau)} \frac{2r^{1+\tau}}{\rho^{1+\tau}(1-\rho^{\tau_{k}})} + \frac{r^{1+\tau}}{\rho} \leq C(\rho)r^{1+\tau},
\end{align*}
by taking $k$ large enough, where we have used
$$ \limsup_{k\rightarrow \infty} k(\tau_{k}-\tau)=0 \ \ \text{and} \ \ \tau_{k} \nearrow \tau ,     $$
in the eighth inequality. Thereby we have shown $ u $ is $C^{1,\tau}$  at $ 0 $ in this case.

$\text{Case 2}$. If $ |Du(0)| > r^{\tau} $, where $ r $ is the same as Case 1. For simplicity, denoting $ |Du(0)|^{\frac{1}{\tau}} $ by $ r_{1} $, and defining
\begin{equation}\label{the:eq15}
u_{r_{1}}(x) := \frac{u(r_{1}x)-u(0)}{r_{1}^{1+\tau}},
\end{equation}
as we discussed in Case 1, we have
\begin{align}\label{the:eq16}
\begin{split}
\sup_{B_{r_{1}}(0)} |u(r_{1}x)-u(0)| & \leq C r_{1}^{1+\tau}\left(1+|Du(0)|r_{1}^{-\tau_{k}}\right)  \leq  C r_{1}^{1+\tau}(1+r_{1}^{\tau-\tau_{k}})
  \leq  C r_{1}^{1+\tau},
\end{split}
\end{align}
where $ \tau_{k} \nearrow \tau $ is used.

From \eqref{the:eq1}, \eqref{the:eq2} and \eqref{the:eq15}, it follows that $ u_{r_{1}} $ is a viscosity sub-solution to
\begin{equation}\label{the:eq17}
\min_{i=0,1,\cdots,M} \bigg \{ - \bigg(|Du_{r_{1}}|^{\widehat{\alpha}_{i}(x)}  + \widehat{a}(x)|Du_{r_{1}}|^{\widehat{\beta_{i}}(x)}       \bigg)    \Delta_{p}^{\rm{N}} u_{r_{1}}\bigg \} = ||\widehat{f}||_{L^{\infty}(B_{1})}
\end{equation}
and it is a viscosity super-solution to
\begin{equation}\label{the:eq18}
\max_{i=0,1,\cdots,M} \bigg \{ - \bigg(|Du_{r_{1}}|^{\widehat{\alpha}_{i}(x)}  + \widehat{a}(x)|Du_{r_{1}}|^{\widehat{\beta_{i}}(x)}       \bigg)  \Delta_{p}^{\rm{N}} u_{r_{1}}\bigg \} = -||\widehat{f}||_{L^{\infty}(B_{1})},
\end{equation}
where
\begin{equation*}
\left\{
     \begin{aligned}
     &  \widehat{\alpha_{i}}(x):= \alpha_{i}(r_{1}x),\ \widehat{\beta_{i}}(x):= \beta_{i}(r_{1}x)  , \    \widehat{a}(x) := a(r_{1}x)r_{1}^{\tau(\widehat{\beta_{i}}(x)-\widehat{\alpha}_{i}(x))} , \\
     &  \widehat{f}(x):=  \max_{i=0,1,\cdots,M}\big\{ r_{1}^{1-\tau(1+\widehat{\alpha}_{i}(x))}\big\}||f||_{L^{\infty}(B_{1})}.           \\
     \end{aligned}
     \right.
\end{equation*}
We observe that
\begin{equation}\label{the:eq19}
u_{r_{1}}(0)=0, \ \ |Du_{r_{1}}(0)|=1, \ \ ||\widehat{f}||_{L^{\infty}(B_{1})} \leq 1.
\end{equation}
Invoking Corollary \ref{Section4:Coro1}, we obtain
\begin{equation}\label{the:eq20}
 u_{r_{1}} \in C^{1,\alpha'}_{loc}(B_{1}) \ \ \text{and} \ \ ||u_{r_{1}}||_{C^{1,\alpha'}(B_{3/4})} \leq C .
 \end{equation}

Noticing that \eqref{the:eq19} and \eqref{the:eq20} allow us to find a $ 0<\gamma_{0} \ll 1    $ such that
\begin{equation*}
\label{eq621}
0< r_{0} \leq |Du_{r_{1}}(x)| \leq r_{0}^{-1}, \ \ \forall x \in B_{\gamma_{0}}(0) \ \ \text{and} \ \ r_{0} \ \ \text{is small}.
\end{equation*}
From such an estimate and $ a_{1} \leq  \widehat{\alpha_{i}}(x) \leq  \widehat{\beta_{i}}(x) \leq a_{2}    $, we arrive
\begin{equation*}
\label{eq622}
0<C_{1} \leq \min_{i=0,1,\cdots,M} \big \{ |Du_{r_{1}}|^{\widehat{\alpha}_{i}(x)}  + \widehat{a}(x)|Du_{r_{1}}|^{\widehat{\beta_{i}}(x)}\big \} \leq C_{2} < +\infty,
\end{equation*}
which together with \eqref{the:eq17} can obtain
\begin{equation*}
\label{eq623}
-\Delta_{p}^{\rm{N}} u_{r_{1}} \leq \frac{||\widehat{f}||_{L^{\infty}(B_{1})}}{\min_{i=0,1,\cdots,M} \big \{ |Du_{r_{1}}|^{\widehat{\alpha}_{i}(x)}  + \widehat{a}(x)|Du_{r_{1}}|^{\widehat{\beta_{i}}(x)}\big \}}  \leq C.
\end{equation*}
Similarly, we also get
\begin{equation*}
\label{eq624}
-\Delta_{p}^{\rm{N}} u_{r_{1}} \geq -C.
\end{equation*}

From the discussions in Remark \ref{rk6}, we have
\begin{equation*}
\mathcal{M}^{+}_{\lambda,\Lambda}(D^{2}u_{r_{1}}) \geq -C \ \ \text{and} \ \  \mathcal{M}^{-}_{\lambda,\Lambda}(D^{2}u_{r_{1}}) \leq C.
\end{equation*}
Resorting to the classical estimate (see, \cite[Theorem 2]{CA89} and \cite[Section 8.2]{CC95}) for the uniformly elliptic equation, we have $ u_{r_{1}} \in C^{1,\widetilde{\alpha}}_{loc}(B_{\gamma_{0}}(0)) $ for every $\widetilde{\alpha} \in (0,\widehat{\beta_{0}})$, with the estimate
$$ \sup_{B_{r}(0)}\big|u_{r_{1}}(x)-u_{r_{1}}(0)-Du_{r_{1}}(0)\cdot x\big| \leq Cr^{1+\widetilde{\alpha}}, \ \forall r\in (0,c_{0}\gamma_{0}) ,       $$
where $ c_{0} >0 $ is a sufficiently small constant. Scaling back to $ u $ and setting $ \widetilde{\alpha} := \tau $, we get that
\begin{equation}\label{the:eq21}
\sup_{B_{r}(0)} \big | u(y)-u(0)-Du(0)\cdot y \big | \leq Cr^{1+\tau}, \ \forall \  r \in (0,c_{0}r_{1}\gamma_{0}).
\end{equation}
Finally, we consider the case $ r \in (c_{0}r_{1}\gamma_{0},r_{1}) $, a direct calculation yields that
\begin{align}\label{the:eq22}
\begin{split}
\sup_{B_{r}(0)} \big |u(y)-u(0)-Du(0)\cdot y \big | & \leq \sup_{B_{r_{1}}(0)} |u(x)-u(0)|+|Du(0)|r_{1}  \\
 &  \overset{\eqref{the:eq16}}{\leq} Cr_{1}^{1+\tau} +r_{1}^{1+\tau} \leq \frac{C+1}{(c_{0}\gamma_{0})^{1+\alpha}}r^{1+\tau}.
\end{split}
\end{align}
Combining \eqref{the:eq21} with \eqref{the:eq22}, we obain that $ u $ is $ C^{1,\tau}$ at $0$ in this case.

This completes the proof of Theorem \ref{thm2}.
\end{proof}

\noindent {\bf Acknowledgments} The authors would like to express their sincere gratitude to the anonymous referee for his/her useful and constructive comments that greatly improved the manuscript.

\vspace{3mm}

\noindent {\bf Author Contributions} We declare that all authors have reviewed and contributed equally to this manuscript.

\vspace{3mm}

\noindent {\bf Funding} This work was supported by the National Natural Science Foundation of China (No. 12271093) and the Jiangsu Provincial Scientific Research Center of Applied Mathematics (Grant No. BK20233002) and the Start-up Research Fund of Southeast University (No. 4007012503).

\vspace{3mm}

\noindent {\bf Data Availability} No datasets were generated or analyzed during the current study.

\vspace{3mm}

\noindent {\bf \Large Declarations}

\vspace{3mm}

\noindent {\bf Competing Interests} The authors declare no competing interests.

\vspace{3mm}

\noindent {\bf Ethical Approval} Not applicable

\end{sloppypar}
\end{document}